\pdfoutput=1 

\documentclass{amsart}%
\usepackage{amsfonts}
\usepackage{amsmath}
\usepackage{amssymb}
\usepackage{graphicx}
\usepackage{version}
\usepackage{caption}%
\theoremstyle{plain}
\newtheorem{theorem}{Theorem}[section]

\newtheorem{lemma}{Lemma}[section]
\newtheorem{proposition}{Proposition}[section]
\theoremstyle{definition}
\newtheorem{definition}{Definition}[section]
\newtheorem{example}{Example}[section]
\newtheorem{remark}{Remark}[section]
\numberwithin{equation}{section}
\excludeversion{comment1}
\begin{document}
\title[Flows on Fractals]{Conjugacies provided by fractal transformations I\ : Conjugate measures,
Hilbert spaces, orthogonal expansions, and flows, on self-referential spaces.}
\author[C. Bandt]{Christoph Bandt}
\address{University of Greifswald}
\author[M. F. Barnsley]{Michael Barnsley}
\address{Australian National University }
\author[M. Hegland]{Markus Hegland}
\address{Australian National University}
\author[A. Vince]{Andrew Vince}
\address{Australian National University\\
 }

\begin{abstract}
Theorems and explicit examples are used to show how transformations between
self-similar sets (general sense) may be continuous almost everywhere with
respect to stationary measures on the sets and may be used to carry well known
flows and spectral analysis over from familiar settings to new ones. The focus
of this work is on a number of surprising applications including (i) what we
call fractal Fourier analysis, in which the graphs of the basis functions are
Cantor sets, being discontinuous at a countable dense set of points, yet have
very good approximation properties; (ii) Lebesgue measure-preserving flows, on
polygonal laminas, whose wave-fronts are fractals. The key idea is to exploit
fractal transformations to provide unitary transformations between Hilbert
spaces defined on attractors of iterated function systems. Some of the
examples relate to work of Oxtoby and Ulam concerning ergodic flows on regions
bounded by polygons.

\end{abstract}
\maketitle

\section{Introduction}

In this paper we provide results and explicit examples to show how
transformations between some fractals, and other self-referential sets, may
both be continuous almost everywhere and map well-known flows and spectral
analysis from familiar settings to new ones. Our focus is on a number of
surprising applications including: (i) what we call "fractal Fourier
analysis", in which the basis functions are discontinuous at a countable dense
set of points of a real interval, yet have good approximation properties; (ii)
Lebesgue measure-preserving flows on tori whose wave-fronts are fractal curves.

The key idea is to exploit fractal transformations to provide unitary
transformations between Hilbert spaces defined on attractors of iterated
function systems. Some of our examples relate to the work of Oxtoby and Ulam
\cite{oxtoby}, concerning ergodic flows on real geometrical domains.

Let $A_{F}$ and $A_{G}$ be non-overlapping attractors of two contractive
iterated function systems (IFSs), $F$ and $G$ respectively. We give conditions
under which the fractal transformation $T_{FG}:A_{F}\rightarrow A_{G}$
(defined in Section \ref{sec2}) is measureable and continuous almost
everywhere with respect to any stationary measure $\mu_{F}$ (defined in
Section \ref{sec2}). We show that $T_{FG}$ yields an isometry $U_{FG}%
:\mathcal{L}^{2}(A_{F},\mu_{F})\rightarrow\mathcal{L}^{2}(A_{G},\mu_{G}))$,
where $\mu_{F}$ and $\mu_{G}$ are a corresponding pair of stationary measures.
If $L_{F}:D_{F}\subset\mathcal{L}^{2}(A_{F},\mu_{F})\rightarrow\mathcal{L}%
^{2}(A_{F},\mu_{F})$ is a linear operator with dense domain $D_{F},$ then
\[
L_{G}:=U_{FG}L_{F}U_{GF}%
\]
is a linear operator on $\mathcal{L}^{2}(A_{G},\mu_{G})$ with dense domain
$T_{FG}(D_{F})$. If $L_{F}$ is self-adjoint, then so is $L_{G}$. In some cases
$\mu_{F}$ is Lebesgue measure on a subset of $\mathbb{R}^{n}$ such as line
segment, a filled triangle, or a cube; and in other cases it a uniform measure
on a fractal such as a Sierpinski triangle. In these cases, familiar
differential and integral equations, including those associated with
Laplacians on post critically finite (p.c.f.) fractals \cite{kigami,
strichartz}, can be transformed to yield interesting counterparts on other
(not necessarily p.c.f.) fractals.

By way of examples (i) we introduce what we call "fractal Fourier analysis",
in which the basis functions are discontinuous at a countable dense set of
points, yet have good approximation properties including overcoming the
edge-effect problem that besets standard Fourier approximation; and (ii) we
introduce and exemplify certain flows on self-similar sets, we provide rough
versions of flows on tori, and we exhibit the solution of a heat equation on a
rough filled triangle, with Dirichlet boundary conditions.

\section{\label{sec2}Fractal transformations and invariant measures}

This section introduces some essential concepts that run throughout the paper,
including the invariant measure of an IFS with probabilities, called a
$p$-measure, and fractal transformations from the attractor of one IFS to the
attractor of another. The main result of this section are Theorem~\ref{thm:ZM}
which states that if an attractor is not equal to its dynamical boundary, then
all $p$-measures of the critical set, the dynamical boundary, and the forward
orbit of overlap set under the IFS (which we call the inner boundary), are
zero; and Theorem~\ref{tauthm} which states that a fractal transformation
between non-overlapping attractors is measurable and continuous almost
everywhere with respect to every $p$-measure, and that a such a fractal
transformation is $p$-measure preserving.

\subsection{Non-Overlapping Attractors and Fractal Transformations}

The purpose of this subsection is to define the central notions of
non-overlapping attractor and fractal transformation from one attractor to another.

Let $\mathbb{N=\{}1,2,3,...\}$ and $\mathbb{N}_{0}=\{0,1,2,...\}$. Throughout
this paper we restrict attention to iterated function systems (IFSs) of the
form%
\[
F=\{X;f_{1},f_{2},...,f_{N}\}
\]
where $N\in\mathbb{N}$ is fixed, $X$ is a complete metric space, and
$f_{i}:X\rightarrow X$ is a contraction for all $i\in I:=\{1,2,...,N\}$. By
\textbf{contraction} we mean there is $\lambda\in\lbrack0,1)$, such that
$d_{\mathbb{X}}(f_{i}(x),f_{i}(y))\leq\lambda d_{\mathbb{X}}(x,y)$ for all
$x,y\in X$, for all $i\in I$.

Define $F^{-1}:2^{A}\rightarrow2^{A}$ and $F:2^{A}\rightarrow2^{A}$ by
\[
F^{-1}(U)=\cup_{i=1}^{N}f_{i}^{-1}(U)\text{ and }F(U)=\cup_{i=1}^{N}%
f_{i}(U)\text{,}%
\]
for all $U\subset A$, where $f_{i}^{-1}(U)=\{x\in A:f_{i}(x)\in U\}$, and
$f_{i}(U)=\{f_{i}(x)\in A:x\in U\}$. Let $F^{-k}$ mean $F^{-1}$ composed with
itself $k$ times, let $F^{k}$ mean $F$ composed with itself $k$ times, for all
$k\in\mathbb{N}$, and let $F^{0}=F^{-0}=I$.

If ${\mathbb{H}}(X)$ denotes the collection of nonempty compact subsets of
$X$, then the classical Hutchinson operator $F\,:\,{\mathbb{H}}(X)\rightarrow
{\mathbb{H}}(X)$ is just the operator $F$ above restricted to ${\mathbb{H}%
}(X)$. According to the basic theory of contractive IFSs as developed in
\cite{hutchinson}, there is unique \textbf{attractor} $A\subset X$ of $F$.
That is, $A\ $is the unique nonempty compact subset of $X$ such that
\[
A=F(A)\text{.}%
\]
The attractor $A$ has the property
\[
A=\lim_{k\rightarrow\infty}F^{k}(S),
\]
where convergence is with respect to the Hausdorff metric and is independent
of $S\in{\mathbb{H}}(X)$.

Since, in this paper, we are only interested in $A$ itself, henceforth let
$X=A$. Moreover, throughout this paper the following assumptions are made:

\begin{itemize}
\item $F=\{A;f_{1},f_{2},...,f_{N}\}$ is an IFS with attractor $A$ and such
that each of its functions is a contraction and is a homeomorphism onto its
image. \vskip2mm
\end{itemize}

(Note that, under these assumptions, $f_{i}^{-1}(S):=\{a\in A:f_{i}(a)\in
S\}=f_{i}^{-1}($ $f_{i}(A)\cap S)$ for all $i$, for all $S\subset A$.)

Let $I=\{1,2,\dots,N\}$, and let $I^{\infty}$, referred to as the \textbf{code
space}, be the set of all infinite sequences $\theta=\theta_{1}\theta
_{2}\theta_{3}\cdots$ with elements from $I$. The \textbf{shift operator}
$S:I^{\infty}\rightarrow I^{\infty}$ is defined by $S(\theta_{1}\theta
_{2}\theta_{3}\cdots)=\theta_{2}\theta_{3}\theta_{4}\cdots$. Define a metric
$d$ on $I^{\infty}=\{1,2,...,N\}^{\infty}$ so that, for $\theta,\sigma\in
I^{\infty}$ with $\theta\neq\sigma$, the distance $d(\theta,\sigma)=2^{-k}$,
where $k$ is the least integer such that $\sigma_{k}\neq\theta_{k}$. The pair
$(I^{\infty},d)$ is a compact metric space.

\begin{definition}
The \textbf{coding map}, $\pi:I^{\infty}\rightarrow A$ is defined by
\[
\pi(\sigma)=\lim_{k\rightarrow\infty}f_{\sigma_{1}}\circ f_{\sigma_{2}}%
\circ...\circ f_{\sigma_{k}}(a)\text{,}%
\]
for any fixed $a\in A$, for all $\sigma=\sigma_{1}\sigma_{2}...\in I^{\infty}$.
\end{definition}

Under the assumption that the IFS is contractive, it is well known that the
limit is a single point, independent of $a\in A$, convergence is uniform over
$I^{\infty}$, and $\pi$ is continuous and onto.

\begin{example}
[The code space IFS]\label{ex:Z} The IFS $Z=\{I^{\infty};\,s_{1},s_{2}%
,\dots,s_{N}\}$, where $s_{i}:I^{\infty}\rightarrow I^{\infty}$ is defined by
$s_{i}(\sigma)=i\,\sigma$, satisfies all the conditions. In particular, the
contraction constant for all $i$ is $\lambda=\frac{1}{2}$. In this case $\pi$
is the identity map on $I^{\infty}$.
\end{example}

\begin{definition}
Define the \textbf{critical set} of $A$ (w.r.t. $F$) to be
\[
C=\bigcup\limits_{i\neq j}f_{i}(A)\cap f_{j}(A)\text{.}%
\]

\end{definition}

Let $\overline{U}$ be the closure of $U\subset A$.

\begin{definition}
Define the \textbf{dynamical boundary} of $A$ (w.r.t. $F$) to be
\[
\partial A=\overline{\bigcup\limits_{k=1}^{\infty}F^{-k}(C)}.
\]

\end{definition}

The notion of the dynamical boundary was introduced by Mor\'{a}n \cite{moran},
in the context of similitudes on $\mathbb{R}^{n}$. In general, $\partial A$ is
not equal to the topological boundary of $A$ (see Example~\ref{ex:db}).

\begin{definition}
For the IFS $F$, we define the \textbf{inner boundary} of the attractor $A$
(w.r.t. $F$) to be
\[
\widehat{C}=\bigcup_{k\in\mathbb{N}_{0}}F^{k}(C).
\]

\end{definition}

The inner boundary of $A$ is the set of points with more than one address: a
proof of the following proposition appears in \cite{kameyama}.

\begin{proposition}
\label{prop:single} $\widehat{C}=\{x\,:\,|\pi^{-1}(x)|\neq1\}$.
\end{proposition}

\begin{definition}
Define $A_{F}$ to be \textbf{non-overlapping} (w.r.t. $F$) when
\[
A\neq\partial A\text{.}%
\]

\end{definition}

\begin{example}
\label{ex:db} Let $F=\{[0,1];\,f_{1},f_{2}\}$, where the metric on the unit
interval $[0,1]$ is the Euclidean metric. Note that the topological boundary
of $[0,1]$ is empty; every point in $[0,1]$ lies in its interior. If
$f_{1}(x)=\frac{1}{2}\,x,\,f_{2}(x)=\frac{1}{2}\,x+\frac{1}{2}$, then the
dynamical boundary of the attractor $A=[0,1]$ is $\partial A=\{0,1\}$. In this
case, by definition, $A$ is non-overlapping. On the other hand, if
$f_{1}(x)=\frac{2}{3}\,x,\,f_{2}(x)=\frac{2}{3}\,x+\frac{1}{3}$, then again
$A=[0,1]$, but $\partial A=[0,1]$. In this case $A$ is overlapping.
\end{example}

We are going to need the following topological lemma, which generalizes a
result in \cite{manifolds}. A point $\omega\in I^{\infty}$ is called
\textbf{disjunctive} if $\left\{  S^{k}\omega:k\in\mathbb{N}\right\}  $ is
dense in $I^{\infty}$.

\begin{lemma}
\label{lemma01} Let $F=\{A;f_{1},f_{2},...,f_{N}\}$ be an IFS with attractor
$A$, and let $\omega\in I^{\infty}$ be disjunctive. We have $\pi(\omega)\in
A\backslash\partial A$ if and only if $A\backslash\partial A\neq\emptyset$.
\end{lemma}

\begin{proof}
We begin with two observations. (i)The set $\partial A$ is closed and
$F^{-1}(\partial A)\subset\partial A$. Hence, if $\theta\in I^{\infty}$ obeys
$\pi(\theta)\in\partial A$, then $\pi(S\theta)\in\partial A$, whence
$\pi(S^{k}\theta)\in\partial A$ for all $k\in\mathbb{N}_{0}$, whence
$\overline{\{\pi(S^{k}\theta)\}_{k=0}^{\infty}}\subset\partial A$. (ii) If
$\omega\in I^{\infty}$ is disjunctive, then, using the continuity of $\pi$,
$\overline{\{\pi(S^{k}\omega)\}_{k=0}^{\infty}}=A$.

Let $\omega\in I^{\infty}$ be disjunctive.

($\Rightarrow$)Suppose that $\pi(\omega)\in A\backslash\partial A$. Then
$A\backslash\partial A\neq\emptyset$.

($\Leftarrow$)Suppose that $A\backslash\partial A\neq\emptyset$. If
$\pi(\omega)\in$ $\partial A$, it follows that $S^{k}\omega\in\partial A$ for
all $k$, so by (i) and (ii), $A=\overline{\{\pi(S^{k}\omega)\}_{k=0}^{\infty}%
}\subset\partial A$; but $\partial A\subset A,$ so $A=\partial A;$ hence
$A\backslash\partial A=\emptyset$, which is not possible, so $\pi(\omega)\in
A\backslash\partial A$.
\end{proof}

The code space $I^{\infty}$ is equipped with the lexicographical ordering, so
that $\theta>\sigma$ means $\theta\neq\sigma$ and $\theta_{k}>\sigma_{k}$
where $k$ is the least index such that $\theta_{k}\neq\sigma_{k}$. Here
$1>2>3\cdots>N-1>N$.

\begin{definition}
A \textbf{section} of the coding map $\pi:I^{\infty}\rightarrow A$ is a map
$\tau:A\rightarrow I^{\infty}$ such that $\pi\circ\tau$ is the identity. In
other words $\tau$ is a map that assigns to each point in $A$ an
\textit{address} in the code space. The \textbf{top section} of $\pi
:I^{\infty}\rightarrow A$ is the map $\tau:A\rightarrow I^{\infty}$ given by
\[
\tau(x)=\max\pi^{-1}(x)
\]
for all $x\in A$, where the maximum is with respect to the lexicographic
ordering. The value $\tau(x)$ is well-defined because $\pi^{-1}(x)$ is a
closed subset of $I^{\infty}$.
\end{definition}

The top section is forward shift invariant in the sense that $S(\tau
(A))=\tau(A)$. See \cite{BV4} for a classification, in terms of masks, of all
shift invariant sections, namely sections such that $S(\tau(A))\subset\tau(A)$.

\begin{definition}
\label{def:FT} Let $A_{F}$ and $A_{G}$ be the attractors, respectively, of
IFSs $F=\{A_{F};f_{1},f_{2},...,f_{N}\}$ and $G=\{A_{G};g_{1},g_{2}%
,...,g_{N}\}$ with the same number of functions. The \textbf{fractal
transformations} $T_{FG}:A_{F}\rightarrow A_{G}$ and $T_{GF}:A_{G}\rightarrow
A_{F}$ are defined (see for example \cite{Tops, BHI}) to be
\[
T_{FG}=\pi_{G}\circ\tau_{F}\qquad\text{ and }\qquad T_{GF}=\pi_{F}\circ
\tau_{G},
\]
where $\tau$ is the top section. If $T_{FG}$ is a homeomorphism, then it is
called a \textbf{fractal homeomorphism}, and in this case $T_{GF}%
=(T_{FG})^{-1}$.
\end{definition}

A more general notion of fractal transformation is similarly defined by taking
$\tau$ to be any shift invariant section; see \cite{BV4}. The following simple
proposition is useful. It is well-known, see for example \cite[Theorem
1]{monthly} and \cite{BHI}, for references and subtler results.

\begin{proposition}
Let IFS $F$ be a non-overlapping with attractor $A$, and let $P_{F}=\{\pi
^{-1}(x)\,:\,x\in A\}$, which is a partition of the code space $I^{\infty}$.
For two non-overlapping IFSs $F$ and $G$, and fractal transformation $T_{FG}$,
if $P_{F}=P_{G}$, then $T_{FG}$ is a homeomorphism.
\end{proposition}

\subsection{Invariant Measures on the Attractor of an IFS}

\label{sec:IM} In this subsection we recall the definition of the invariant
measures on an IFS with probabilities, also called $p$-measures, and determine
that the dynamical boundary of the attractor $A$ and a certain subset of $A$
associated with the critical set of $A$, that we call the inner boundary, have
measure zero.

\begin{definition}
Let $p=(p_{1},p_{2},...,p_{N})$ satisfy $p_{1}+p_{2}+...+p_{N}=1$ and
$p_{i}>0$ for $i=1,2,...,N$. Such a positive $N$-tuple $P$ will be referred to
as a \textbf{probability vector}. It is well known that there is a unique
normalized positive Borel measures $\mu$ supported on $A$ and invariant under
$F$ in the sense that%
\begin{equation}
\mu(B)=\sum_{i=1}^{N}p_{i}\,\mu(f_{i}^{-1}(B)) \label{eq:invariance}%
\end{equation}
for all Borel subsets $B$ of $X$. We call $\mu$ the \textbf{invariant measure
of} $F$ corresponding to the probability vector $p$ and refer to it as the
$p$\textbf{-measure} (w.r.t. $F$). To emphasize the dependence on $p$, we may
write $\mu_{p}$ in place of $\mu$.
\end{definition}

\begin{example}
\label{ex:mZ} This is a continuation of Example~\ref{ex:Z}, where
$Z=\{I^{\infty};\,s_{1},s_{2},\dots,s_{N}\}$. For a probability vector
$p=(p_{1},p_{2},...,p_{N})$, the corresponding $p$-measure is the Bernoulli
measure $\nu_{p}$ where
\[
\nu_{p}\left(  [\sigma_{1}\,\sigma_{2}\cdots\sigma_{n}]\right)  =\prod
_{i=1}^{n}p_{\sigma_{i}},
\]
where $[\sigma_{1}\,\sigma_{2}\cdots\sigma_{n}]:=\{\omega\in I^{\infty}:$
$\omega_{i}=\sigma_{i}$ for $i=1,2,..,N\}$ denotes a cylinder set, the
collection of which generate the sigma algebra of Borel sets of $I^{\infty}$.
\end{example}

The following known result, see for example \cite[statement and proof of
Theorem 9.3]{falconer}, is relevent to the present work.

\begin{proposition}
\label{prop:similitudes} If $F$ consists of similitudes with scaling ratio of
$f_{i}$ equal to $c_{i}<1$, and obeys the open set condition, and if the
probabilities are chosen such that $p_{i}=c_{i}^{D}$, where $D$ is the
Hausdorff dimension of $A$, then $\mu_{p}$ is equal to the Hausdorff measure
on $A$.
\end{proposition}

The Hausdorff measure prescribed in Proposition \ref{prop:similitudes} is
sometimes referred to as the \textit{uniform measure} on the attractor.

The following result is proved in \cite{hutchinson}.

\begin{lemma}
\label{lem:ZF} If $F$ is an IFS with probability vector $p$, corresponding
invariant measure $\mu_{p}$, and $Z$ is the IFS of Example~\ref{ex:Z} with the
same probability vector $p$ and corresponding invariant measure $\nu_{p}$,
then
\[
\mu_{p}(B)=\nu_{p}(\pi_{F}^{-1}(B)
\]
for all Borel sets $B$.
\end{lemma}

The following theorem relates the topological concept of non-overlapping to
the $p$-measures of the dynamical boundary and the inner boundary. It can be
viewed as an extension of a result of Bandt and Graf \cite{BG}, who show that
the Hausdorff measure of the critical set of the attractor of an IFS of
similitudes in $\mathbb{R}^{n},$ that obeys the OSC, is zero.

\begin{theorem}
\label{thm:ZM} Let $F=\{A;f_{1},f_{2},...,f_{N}\}$ be an IFS (with
probabilities $p$) with attractor $A$, invariant measure $\mu_{p},$ dynamical
boundary $\partial A,$ and inner boundary $\widehat{C}$. Let $\mu_{p}$ be an
invariant measure for $F$. If $A$ is non-overlapping then, for all probability
vectors $p$,

(i) $\mu_{p}(A\backslash\partial A)=1$;

(ii)$\ \mu_{p}(\widehat{C})=0$.
\end{theorem}

\begin{proof}
To simplify notation let $p$ be any probability vector, let $\mu=\mu_{p}$, and
let $v=v_{p}$, the $p$-measure on $I^{\infty}$ introduced in Examples
\ref{ex:Z} and \ref{ex:mZ}.

Proof of (i): Let $D\subset I^{\infty}$ be the set of disjunctive points. If
$A$ is non-overlapping then, by Lemma \ref{lemma01}, $\pi(D)\subset
A\backslash\partial A$.

Hence
\[
1\geq\mu(A\backslash\partial A)\geq\mu(\pi(D))=v(D)=1\text{,}%
\]
where we have used Lemma \ref{lem:ZF} and the fact that $v(D)=1$ (for all
vectors $p$), see \cite{staiger}.

Proof of (ii): Let $C$ be the critical set of $A$. It follows from (1) that
$\mu(F^{-1}(C))=0$ and therefore $\mu(f_{i}^{-1}(C))=0$ for all $i$. By the
invariance property
\[
\mu(C)=\sum\limits_{i=1}^{N}p_{i}\mu(f_{i}^{-1}(C))=0.
\]
Now, for each $j$,
\begin{align*}
\mu(f_{j}(C))  &  =\sum\limits_{i=1}^{N}p_{i}\mu(f_{i}^{-1}(f_{j}%
(C)))=p_{j}\mu(C)+\sum\limits_{i\neq j}p_{j}\mu(f_{i}^{-1}(f_{j}(C)))\\
&  =\sum\limits_{i\neq j}p_{j}\mu(f_{i}^{-1}(f_{j}(C)))\leq\sum\limits_{i\neq
j}p_{j}\mu(f_{i}^{-1}(C))=0,
\end{align*}
the inequality for the following reason: since $f_{i}^{-1}(S)=$ $f_{i}%
^{-1}(f_{i}(A)\cap S)$, for all $S\subset A$, we have that $f_{i}^{-1}%
(f_{j}(C))\subset f_{i}^{-1}\left(  f_{i}(A)\cap f_{j}(A)\right)  \subset
f_{i}^{-1}(C)$, and the last equality because $\mu(f_{i}^{-1}(C)=0$. Since
this is true for all $i$, we have $\mu(F(C))=0$. Induction can now be used,
similarly, to show that $\mu(F^{k}(C))=0$ for all $k\in\mathbb{N}_{0}$. This
suffices to prove (2) in the statement of the theorem.
\end{proof}

\begin{remark}
By Theorem~\ref{thm:ZM}, the definition of non-overlapping, i.e., $\partial
A\neq A$, is independent of the probability vector $p$. Also, if an IFS is
non-overlapping, then whether or not $\mu_{p}(C)=0$ is independent of $p$.
Also, if
\begin{equation}
\overline{\bigcup\limits_{k=1}^{\infty}F^{-k}(C)}={\bigcup\limits_{k=1}%
^{\infty}F^{-k}(C)},\label{eq:equiv}%
\end{equation}
which occurs for example if $A$ is p.c.f., then the converse to
Theorem~\ref{thm:ZM} holds, namely, if $\mu_{p}(C)=0$ for any probability
vector $p$, then $A$ is non-overlapping. In particular if Equation
\ref{eq:equiv} holds, then whether or not $\mu_{p}(C)=0$ is independent of the
probability vector $p$.
\end{remark}

The proof of the following theorem appears in \cite[Theorem 2.1]{MR}, which
also states that, under the assumption of the open set condition (OSC),
whether or not $\mu_{p}(C)=0,$ is independent of $p$; but that theorem applies
only to an IFS consisting of similitudes.

\begin{theorem}
\label{thm:DB} Let $F$ be a contractive IFS of similitudes on $\mathbb{R}^{n}%
$, that obeys set condition. If $C$ is the critical set, then $\mu_{p}(C)=0$
for all $p$-measures $\mu_{p}$ (w.r.t. $F$).
\end{theorem}

\subsection{Continuity and Measure Preserving Properties of Fractal
Transformations}

\label{sec:MP}

The main results of this subsection are that fractal transformations between
non-overlapping attractors are measurable, continuous almost everywhere, and
map $p$-measures to $p$-measures.

\begin{theorem}
\label{tauthm}Let $F=\{A;f_{1},f_{2},...,f_{N}\}$ be an IFS with
non-overlapping attractor $A$ and invariant measure $\mu$. The top section of
$\tau:A\rightarrow I^{\infty}$ is measureable and continuous almost everywhere
w.r.t. $\mu$, for all $p$.
\end{theorem}

\begin{proof}
We first prove that $\tau:$ $A\rightarrow I^{\infty}$ is measureable by
showing that $\tau_{F}$ is the uniform limit of a sequence of simple functions
whose maximal sets upon which $\tau$ has constant value are Borel sets. Define
the sequence of simple functions $\tau^{(k)}:A\rightarrow I^{\infty}$ for
$k\in\mathbb{N}$ by
\[
\tau^{(k)}(x)=\tau(x)|_{k}\overline{1}%
\]
for all $x\in A$, where $\overline{1}:=111\cdots$ and $\sigma|_{k}:=\sigma
_{1}\cdots\sigma_{k}$. The sequence $\{\tau^{(k)}\}_{k\in\mathbb{N}}$
converges uniformly to $\tau$ because $d(\tau^{(k)}(x),\tau(x))\leq2^{-k}$; in
fact $\tau(x)=\sup\{\tau^{(k)}(x):k\in\mathbb{N\}}$. To show that $\tau$ is
measurable, it now suffices to show that the maximal subsets of $A$ on which
$\tau^{(k)}(x)$ is constant, namely
\[
D_{\sigma_{1}...\sigma_{k}}:=\{x\in A:\tau^{(k)}(x)=\sigma_{1}...\sigma
_{k}\overline{1}\}\text{,}%
\]
are Borel sets. This is established by showing, by induction, that
\[
D_{\sigma_{1}...\sigma_{k}}:=f_{\sigma_{1}}\circ f_{\sigma_{2}}\circ\dots\circ
f_{\sigma_{k}}(A)\backslash\{f_{\theta_{1}}\circ f_{\theta_{2}}\circ\dots\circ
f_{\theta_{k}}(A):\theta_{1}...\theta_{k}<\sigma_{1}...\sigma_{k}\}.
\]
That is, the largest set on which $\tau^{(k)}(x)$ is constant is exactly
$D_{\pi(x)|k}$. Each of the sets $f_{\theta_{1}}\circ f_{\theta_{2}}\circ
\dots\circ f_{\theta_{k}}(A)$ is a Borel set, so $D_{\sigma_{1}...\sigma_{k}}$
is too.

To prove continuity, let $D=A\setminus\widehat{C}$, which is, by
Proposition~\ref{prop:single} is the set of points with exactly one address.
Let $x\in D$ and assume, by way of contradiction, that there is a sequence of
points $\{x_{n}\}$ such that $x_{n}\rightarrow x$, but $\tau(x_{n}%
)\nrightarrow\tau(x)$. Using the notation $\sigma:=\tau(x)$ and $\omega
_{n}:=\tau(x_{n})$, we have $x_{n}\rightarrow x$, but $\omega_{n}%
\nrightarrow\sigma$. Since code space is compact, by going to a subsequence if
needed, we may assume that $\omega_{n}\rightarrow\omega\neq\sigma$. Now
\[
\pi(\sigma)=\pi\circ\tau(x)=x=\lim_{n\rightarrow\infty}x_{n}=\lim
_{n\rightarrow\infty}\pi\circ\tau(x_{n})=\lim_{n\rightarrow\infty}\pi
(\omega_{n})=\pi(\omega),
\]
the last equality following from the continuity of the coding map $\pi$. This
implies that $\omega\neq\sigma$ are both addresses of $x$, which is a
contradiction because $x\in D$ has exactly one address.
\end{proof}

For an IFS $F$, let
\[
\Gamma_{F}=\pi_{F}^{-1}(\widehat{C}_{F}).
\]
Consider two non-overlapping IFSs $F$ and $G$ with the same probability
vector. With notation as in the Definition~\ref{def:FT} of fractal
transformation, let
\[
\begin{aligned}\Gamma_{\{F,G\}} &= \Gamma_F \cup \Gamma_G \\
\Lambda_{\{F,G\}} &= I^{\infty} \setminus \Gamma_{\{F,G\}} \\
A_F^0 &= \pi_F (\Lambda_{\{F,G\}}  ) \qquad \text{and} \qquad
A_G^0 =  \pi_G (\Lambda_{\{F,G\}}  ) \\
A_F^1 &= A_F \setminus A_F^0 \qquad \quad \text{and} \qquad
A_G^1 = A_G \setminus A_G^0
\end{aligned}
\]
Note that $A_{F}^{0}$ depends also on $G$ and that $A_{G}^{0}$ depends also on
$F$; similar for $A_{F}^{1}$ and $A_{G}^{1}$.

\begin{lemma}
\label{lem:bi} With notation as above

\begin{enumerate}
\item $\mu_{F}(A_{F}^{1}) = \mu_{G}(A_{G}^{1}) = 1$,

\item The fractal transformation $T_{FG}$ maps $A_{F}^{1}$ bijectively onto
$A_{G}^{1}$, and maps $A_{F}^{0}$ into $A_{G}^{0}$.

\item Restricted to $A_{F}^{1}$ we have $(T_{FG})^{-1} = T_{GF}$; hence
$(T_{FG})^{-1} = T_{GF}$ almost everywhere.
\end{enumerate}
\end{lemma}

\begin{proof}
Using Lemma~\ref{lem:ZF} and Theorem~\ref{thm:ZM} we have $\mu(\Gamma_{F}%
)=\mu(\pi_{F}^{-1}\widehat{C}_{F})=\mu_{F}(\widehat{C}_{F})=0$. This implies
that $\mu(\Gamma_{\{F,G\}})=0$ or $\mu(\Lambda_{\{F,G\}})=1$. Again using
Lemma~\ref{lem:ZF} we have $\mu_{F}(A_{F}^{1})=\mu_{F}(\pi_{F}(\Lambda
_{\{F,G\}}))=\mu(\pi_{F}^{-1}\pi_{F}(\Lambda_{\{F,G\}}))\geq\mu(\Lambda
_{\{F,G\}})=1$. This proves statement (1).

Concerning statement (2), by Proposition~\ref{prop:single}, we know that
$\pi_{F}^{-1}=\tau_{F}$ is single-valued on $A_{FG}$. Now $\tau_{F}$ takes
$A_{F}^{1}$ bijectively onto $\Lambda_{\{F,G\}})$ and $\pi_{G}$ takes
$\Lambda_{\{F,G\}}$ bijectively onto $A_{G}^{1}$. Similarly, $\tau_{F}$ takes
$A_{F}^{0}$ into $\Gamma_{\{F,G\}})$ and $\pi_{G}$ takes $\Gamma_{\{F,G\}}$
into $A_{G}^{0}$.

Concerning statement (3), restricted to $A_{GF}$ we have $T_{FG} \circ T_{GF}
= \pi_{G}\circ(\tau_{F} \circ\pi_{F} ) \circ\tau_{G} = \pi_{G}\circ\tau_{G} =
I$, the identity.
\end{proof}

\begin{theorem}
\label{thm:MP} Assume that both $A_{F}$ and $A_{G}$ are non-overlapping, and
let $\mu_{F}$ and $\mu_{G}$ be invariant measures associated with the same
probability vector. Then

\begin{enumerate}
\item $T_{FG}:A_{F}\rightarrow A_{G}$ is measurable and continuous a.e. with
respect to $\mu_{F}$;

\item $\mu_{F}\circ T_{GF}=\mu_{G}$ and $\mu_{G}\circ T_{FG}=\mu_{F}$.
\end{enumerate}
\end{theorem}

\begin{proof}
Since $T_{FG} = \pi_{G}\circ\tau_{F}$, statement (1) follows from the
continuity of $\pi_{G}:I^{\infty}\rightarrow A_{G}$ and Theorem \ref{tauthm}.

Concerning statement (2), let $B$ be a Borel set in $A_{G}$, and let $B^{0} =
B \cap A_{G}^{0}, \, B^{1} = B \cap A_{G}^{1}$. By Lemma~\ref{lem:ZF} and
Lemma~\ref{lem:bi}
\[
\mu_{G}(B) = \mu(\pi_{G}^{-1}B) = \mu(\pi_{G}^{-1} (B^{0} \cup B^{1}) )=
\mu(\pi_{G}^{-1} B^{0}) + \mu(\pi_{G}^{-1} \, B^{1}). = \mu(\tau_{G} B^{1}),
\]
the last equality because $\pi_{G}^{-1} (B^{0}) = \tau_{G} ( B^{0})$, which
has measure zero.

By similar arguments
\[
\mu_{F}(T_{GF}\,B)=\mu_{F}(T_{GF}(B^{0}\cup B^{1}))=\mu_{F}(T_{GF}\,B^{0}%
)+\mu_{F}(T_{GF}\,B^{1})=\mu(\pi_{F}^{-1}\circ\pi_{F}\circ\tau_{G}(B^{1}%
))=\mu(\tau_{G}\,B_{1}),
\]
the second to last equality because $T_{GF}(B^{0})\subset A_{F}^{0}$, which
has measure zero.
\end{proof}

\section{Examples of Fractal Transformations}

\label{sec:examples}

\begin{example}
(Koch curve) \label{ex1}
\end{example}

Let
\begin{align*}
F &  =\{\mathbb{R};f_{1}=\frac{1}{2}-\frac{x}{2},f_{2}=1-\frac{x}{2}\},\\
G &  =\{\mathbb{R}^{2};g_{1}=(\frac{x}{2}+\frac{y}{2\sqrt{3}}-1,\frac
{x}{2\sqrt{3}}-\frac{y}{2}),g_{2}=(\frac{x}{2}-\frac{y}{2\sqrt{3}}+1,-\frac
{x}{2\sqrt{3}}-\frac{y}{2})\}.
\end{align*}
Then $A_{F}=[0,1]$ while $A_{G}$ is a segment of a Koch snowflake curve. In
this case both $T_{FG}$ and $T_{GF}$ are homeomorphisms, because
\[
\{\pi_{F}^{-1}(x):x\in A_{F}\}=\{\pi_{F}^{-1}(x):x\in A_{F}\}\text{.}%
\]
Also%
\[
T_{FG}=T_{GF}^{-1}\text{.}%
\]
If $p_{1}=p_{2}=0.5,$ then $\mu_{F}$ is uniform Lebesgue measure on $[0,1]$.
The pushfoward of $\mu_{F}$ to $A_{G}$ under $T_{FG}$ is the uniform measure
$\mu_{G}$ on $A_{G}$ that uniquely obeys $\mu_{G}(\mathcal{B})=$ $(\mu
_{G}(g_{1}^{-1}(\mathcal{B)})+\mu_{G}(g_{2}^{-1}(\mathcal{B)}))/2$ for all
Borel subsets $\mathcal{B}$ of $A_{G}$. (We remark that the measure of any
Borel subset $\mathcal{B}$ of $A_{G}$ may be computed by, and thought of in
terms of, the chaos game algorithm on $G$ with equal probabilities,
\cite{elton}.) The Hausdorff dimensions of $A_{F}$ and $A_{G}$ are $1$ and
$2\ln2/\ln3$, respectively: thus, a fractal transformation may change the
dimension of a set upon which it acts.

\begin{example}
[Length preserving fractal transformation of the unit interval]Let $F =\{
([0,1] ; \, f_{1}, f_{2} \}$ and $G =\{ ([0,1] ; \, g_{1}, g_{2} \}$, where
\[
\begin{aligned} f_1(x) &= r\,x, & f_2(x) &= (1-r)x + r \\ g_1(x) &= r\, x + (1-r), & g_2(x) &= (1-r)\,x, \end{aligned}
\]
and $0 < r < 1$. The probability vector is $p = (r, 1-r)$, so that the
invariant measure for both $F$ and $G$ is Lebesque measure. By
Theorem~\ref{thm:MP}, the fractal transformation $T_{FG} : [0,1]
\rightarrow[0,1]$ preserves length. This example can be generalized from $2$
to $N$ functions as long as the scaling factors of $f_{i}$ and $g_{i}$ are the
same, say $r_{i}$, for all $i$, and the probability vector $p = (p_{1}, p_{2},
\dots, p_{N})$ satisfies $p_{i} = r_{i}$ for all $i$.
\end{example}

\begin{example}
[Self mappings of the interval]\label{ex:1} If
\begin{align*}
F &  =\left\{  \mathbb{R};\;f_{1}=\frac{x}{2},\;f_{2}=\frac{x}{2}+\frac{1}%
{2}\right\}  ,\\
G_{1} &  =\left\{  \mathbb{R};\;g_{1}=-\frac{x}{2}+\frac{1}{2},\;g_{2}%
=\frac{x}{2}+\frac{1}{2}\right\}  ,\\
G_{2} &  =\left\{  \mathbb{R};\;g_{1}=-\frac{x}{2}+\frac{1}{2},\;g_{2}%
=-\frac{x}{2}+1\right\}  ,\\
G_{3} &  =\left\{  \mathbb{R};\;g_{1}=\frac{x}{2},\;g_{2}=-\frac{x}%
{2}+1\right\}  ,
\end{align*}
then $A_{F}=A_{G_{i}}=[0,1]$ for $i=1,2,3$. All three fractral transformations
$T_{FG_{i}},\,i=1,2,3$, are continuous at all points of $A_{F}^{1}%
=[0,1]\backslash\widehat{C}$ where $\widehat{C}$ is the diadic set
\[
\widehat{C}=\left\{  \frac{k}{2^{n}}:k=0,1,...,2^{n};n\in\mathbb{N}\right\}
\text{.}%
\]
Indeed, $T_{FG_{i}},\,i=1,2,3$, is a homeomorphism when restricted to
$[0,1]\backslash\widehat{C}$. Moreover, $T_{FG_{i}},\,i=1,2,3$, are continuous
from the left at all points in $(0,1]$. If we choose $p_{1}=p_{2}=0.5$, then
the measures $\mu_{F}=\mu_{G_{i}},\,i=1,2,3$, are all the Lebesgue measure on
$[0,1]$. The graph of the function $T_{FG_{1}}$ appears in Figure \ref{tfg02},
and the graph of $T_{FG_{2}}$ appears in Figure \ref{2ndfsin01}.

It can be shown by a symmetry argument that $T_{FG_{2}}$ is its own inverse,
i.e., $T_{FG_{2}}\circ T_{FG_{2}}=id,$ the identity, a.e. This is not obvious
from the definition of $T_{FG_{2}}$ which can be stated by expressing
$x\in\lbrack0,1]$ in binary representation: if
\[
x=\sum_{n=1}^{\infty}d_{n}/2^{n},\quad d_{n}\in\{0,1\},
\]
then
\[
T_{FG_{2}}(x)=\sum_{n=1}^{\infty}(-1)^{n-1}(d_{n}+1)/2^{n}.
\]

\end{example}

%

\begin{figure}[ptb]%
\centering
\fbox{\includegraphics[
natheight=3.413400in,
natwidth=3.413400in,
height=2.0358in,
width=2.0358in
]%
{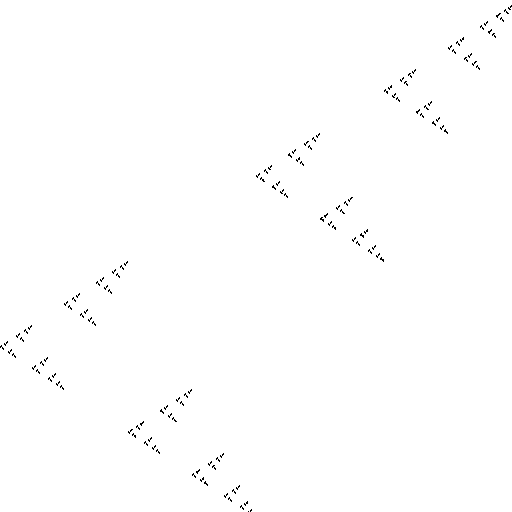}%
}\caption{Graph of the fractal transformation $T_{FG_{1}}$ discussed in
Example \ref{ex:1}. This provides a Lebesgue measure-preserving transformation
on [0,1] that is continuous a.e. but has a dense countable set of
discontinuities. This transformation, and others like it, provide unitary
transformations on $\mathcal{L}^{2}[0,1]$ and "fractal Fourier series", see
Section \ref{sec:unitary}. The viewing window is slightly larger than
[0,1]$\times$[0,1].}%
\label{tfg02}%
\end{figure}
\begin{figure}[ptb]%
\centering
\fbox{\includegraphics[
natheight=3.413400in,
natwidth=3.413400in,
height=2.0358in,
width=2.0358in
]%
{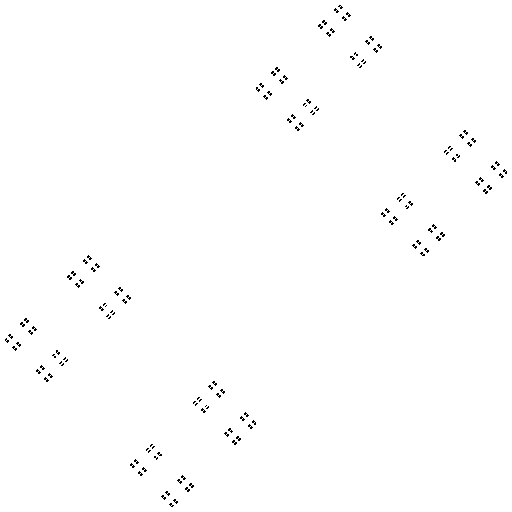}%
}\caption{Graph of the fractal transformation $T_{FG_{2}}$ discussed in
Section \ref{ex:1}. Unlike $T_{FG_{1}}$ in Figure \ref{tfg02}, $T_{FG_{2}}$ is
its own inverse. }%
\label{2ndfsin01}%
\end{figure}

\begin{example}
[Hilbert's space filling curve]\label{ex:Hilbert}

Space filling curves, from the point of view of IFS theory, have been
considered in \cite{sagan}. In \cite{monthly} it is shown how, as follows,
functions such as the Hilbert mapping $h:[0,1]\rightarrow$ $[0,1]^{2}$ (see
Figure \ref{image}) are examples of fractal transformations.

Let $A=A_{1}=(0,0),B=B_{4}=(1,0),C=C_{3}=(1,1),D=D_{2}=(0,1),$ $B_{1}%
=A_{2}=(0,0.5),$ $C_{1}=B_{2}=A_{3}=D_{4}=(0.5,0.5),$ $D_{1}=C_{4}=(0.5,0),$
$C_{2}=D_{3}=(0.5,1),$ and $B_{3}=A_{4}=(1,0.5)$. Let
\begin{align*}
F  &  =\left\{  \mathbb{R};f_{i}=\frac{x+i-1}{4},i=1,2,3,4\right\}  ,\\
G  &  =\left\{  \mathbb{R}^{2};g_{i},i=1,2,3,4\right\}
\end{align*}
where $g_{i}:\mathbb{R}^{2}\rightarrow\mathbb{R}^{2}$ is the unique affine
transformation such that $g_{i}(ABCD)=A_{i}B_{i}C_{i}D_{i}$, by which we mean
$g_{i}(A)=A_{i}$, $g_{i}(B)=B_{i},g_{i}(C)=C_{i},g_{i}(D)=D_{i}$ for
$i=1,2,3,4$. (Similar notation will be used elsewhere in this paper.) The
Hilbert mapping is $h=T_{FG}:[0,1]\rightarrow\lbrack0,1]^{2},$ The functions
in $G$ were chosen to conform to the orientations of Figure \ref{image}, which
comes from Hilbert's paper \cite{hilbert} concerning Peano curves. One way to
prove that $T_{FG}$ is continuous is by using the standard theory of fractal
transformations; see for example \cite[Theorem 1]{monthly}.

If $p_{1}=p_{2}=p_{3}=p_{4}=0.25,$ then the associated invariant measure
$\mu_{F}$ is the Lebesgue measure on $[0,1]$, and $\mu_{G}$ is Lebesgue
measure on $[0,1]^{2}$. The inverse of $T_{FG}^{-1}$ is the fractal
transformation $T_{GF}:[0,1]^{2}\rightarrow\lbrack0,1]$, which is continuous
almost everywhere with respect to two dimensional Lebesgue measure. More
precisely, $T_{GF}\circ h(x)=x$ for almost all $x\in\lbrack0,1]$ (with respect
to Lebegue measure), and $h\circ T_{GF}(x)=x$ for all $x\in\lbrack0,1]^{2}$.
By Theorem~\ref{thm:MP}, the fractal transformation $h$ is Lebesque measure
preserving in that the $2$-dimensional Lebesque measure of the image $h(B)$ of
$B$ equals the $1$-dimensional Lebesque measure of $B$, for any Borel set
$B$.
\begin{figure}[ptb]%
\centering
\includegraphics[
natheight=0.953100in,
natwidth=2.352000in,
height=1.452in,
width=3.5405in
]%
{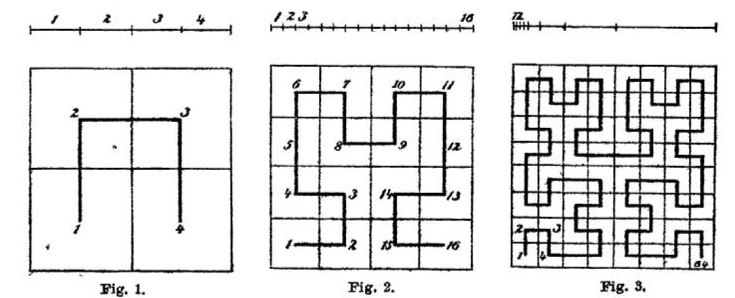}%
\caption{Hilbert's original design for a continuous map from [0,1] to
[0,1]$\times$[0,1].}%
\label{image}%
\end{figure}

\end{example}

\begin{example}
[Fractal transformations between the unit interval and a filled triangle]%
\label{ex4}

Let $A,B,C$ be non-colinear points in $\mathbb{R}^{2}$ and let $D$ be the
mid-point of the line segment $CA$. Let
\begin{align*}
F  &  =\left\{  \mathbb{R};f_{1}(x)=\frac{1}{2}\,x,\;f_{2}(x)=\frac{1}%
{2}\,x+\frac{1}{2}\,\right\}  ,\\
G  &  =\left\{  \mathbb{R}^{2};g_{1},\;g_{2}\right\}  ,
\end{align*}
where $g_{1}$ and $g_{2}$ are the unique affine maps on $\mathbb{R}^{2}$ such
that $g_{1}(ABC)=ADB$ and $g_{2}(ABC)=BDC$, respectively. The unique attractor
of $F$ is $A_{F}=$ $[0,1]\subset\mathbb{R}$, and the unique attractor of $G$
is and $A_{G}=\triangle$, the filled triangle with vertices at $ABC$. If
$p_{1}=p_{2}=0.5$ then $\mu_{F}$ is Lebesgue measure on $[0,1]$, and $\mu_{F}$
is Lebesgue measure on $\triangle$, It readily follows from \cite[Theorem
1]{monthly} that $T_{FG}:[0,1]\rightarrow\triangle$ is continuous and
$T_{FG}([0,1])=\triangle$. It is also readily shown that $T_{GF}%
:\triangle\rightarrow\lbrack0,1]$ is continuous almost everywhere with respect
to two-dimensional Lebesgue measure, with discontinuities located on a
countable set of boundaries of triangles. We have that $T_{FG}\circ
T_{GF}(x)=x$ for all $x\in\triangle$, and $T_{GF}\circ T_{FG}(x)=x$ for almost
all $x\in\lbrack0,1],$ with respect to one-dimensional Lebesgue measure. We
also have $T_{GF}(\triangle)\neq\lbrack0,1]$ but $\overline{T_{GF}(\triangle
)}=[0,1]$.
\end{example}

\begin{example}
[A family of fractal homeomorphisms on a triangular laminar]%
\label{ex:triangles} Let $\bigtriangleup$ denote a filled equilateral triangle
as illustrated in Figure~\ref{fig:triangles}. The IFS $F_{r},\;0<r\leq\frac
{1}{2},$ on $\bigtriangleup$ consists of the four affine functions as
illustrated in the figure on the left, where $\bigtriangleup$ is mapped to the
four smaller triangles so that points $A,B,C$ are mapped are mapped,
respectively, to points $a,b,c$. A probability vector is associated with $F$
such that the probability is proportional to the area of the corresponding
triangle. The IFS $G_{\lambda}$ is defined in exactly the same way, but
according to the figure on the right. The attractor of each IFS is
$\bigtriangleup$. (It is quite a subtle point, that there exists a metric,
equivalent to the Euclidean metric on $\mathbb{R}^{2},$ such that both IFSs
are contractive, see \cite{Akins}.) It is proved in \cite{BHR} that the
corresponding invariant measures $\mu_{F}$ and $\mu_{G}$ are both
2-dimensional Lebesque measure. By Theorem~\ref{thm:MP} and \cite[Theorem
1]{monthly}, or by \cite{BHR}$,$ the fractal transformation $T_{FG}^{r}$ is an
area-preserving homeomorphism of $\bigtriangleup$ for all $0<r\leq\frac{1}{2}%
$. See \cite{BHR} for related examples of volume-preserving fractal
homeomorphisms between tetrahedra.
\end{example}

%

\begin{figure}[ptb]%
\centering
\includegraphics[
trim=0.000000in 2.254340in 0.000000in 0.322332in,
natheight=4.966600in,
natwidth=7.019700in,
height=1.5177in,
width=4.4045in
]%
{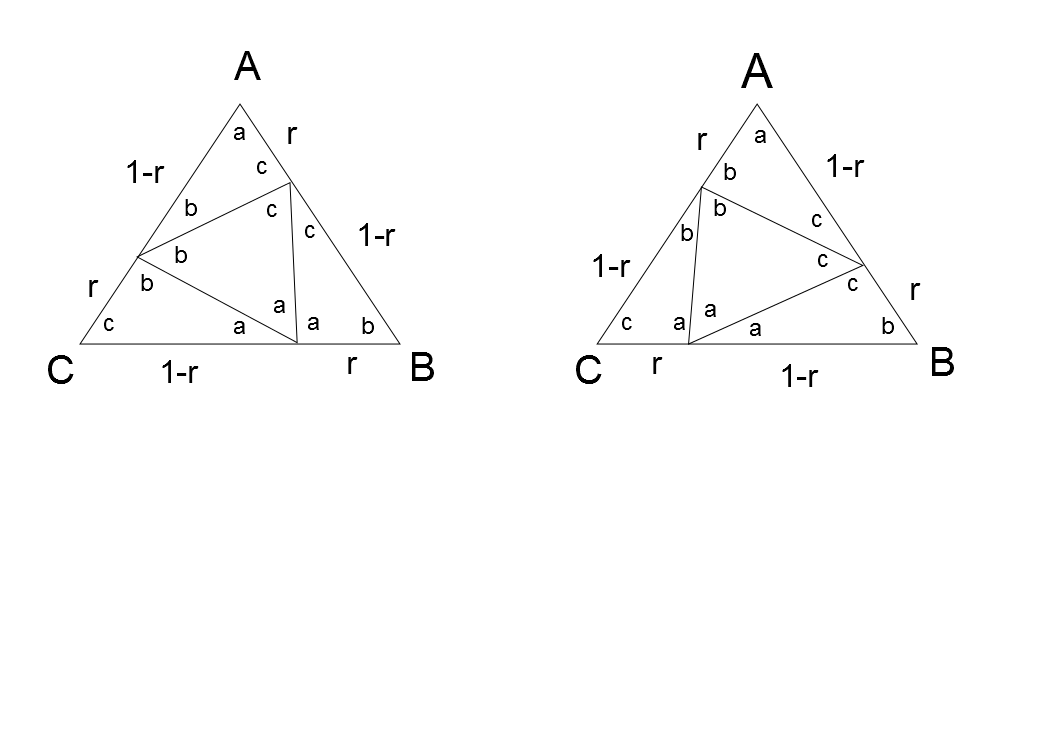}%
\caption{See Example \ref{ex:triangles}.}%
\label{fig:triangles}%
\end{figure}

\section{\label{sec:unitary} Isometries between Hilbert Spaces}

Given an IFS $F$ with attractor $A_{F}$ and an invariant measure $\mu_{F}$,
the Hilbert space $L_{F}^{2}=L^{2}(A_{F},\mu_{F})$ of complex-valued functions
on $A_{F}$ that are square integrable w.r.t. $\mu_{F}$ are endowed with the
inner product $<\cdot,\cdot>_{F}$ defined by
\[
\langle\psi_{F},\varphi_{F}\rangle_{F}=\int\limits_{A_{F}}\overline{\psi_{F}%
}\varphi_{F}\,d\mu_{F}\text{,}%
\]
for all $\psi_{F},\varphi_{F}\in L_{F}^{2}$. Functions that are
\textit{equivalent}, i.e., equal almost everywhere, will be considered the
same function in $L_{F}^{2}$.

\begin{definition}
\label{def:isometry} Given two IFSs $F$ and $G$ with the same number of
functions, with the same probabilities, with attractors $A_{F}$ and $A_{G}$
and invariant measures $\mu_{F}$ and $\mu_{G}$, respectively, let $T_{FG}$ and
$T_{GF}$ be the  fractal transformations. The \textbf{induced isometries}
$U_{FG}:L_{F}^{2}\rightarrow L_{G}^{2}$ and $U_{GF}:L_{G}^{2}\rightarrow
L_{F}^{2}$ are given by
\[
\begin{aligned} (U_{FG}\varphi_{F})(y)&=\varphi_{F}(T_{GF}(y)) \\ (U_{GF}\varphi_{G})(x)&=\varphi_{G}(T_{FG}(x)) \end{aligned}
\]
for all $x\in A_{F}$ and all $y\in A_{G}$. That these linear operators are
isometries is proved as part of Theorem~\ref{thm:isometry} below.
\end{definition}

\begin{theorem}
\label{thm:isometry} Under the conditions of Definition~\ref{def:isometry},

\begin{enumerate}
\item $U_{FG}:L_{F}^{2}\rightarrow L_{G}^{2}$ and $U_{GF}:L_{G}^{2}\rightarrow
L_{F}^{2}$ are isometries;

\item $U_{FG} \circ U_{GF}=id_{F}$ and $U_{GF}\circ U_{FG}=id_{G},$ the
identity maps on $L_{F}^{2}$ and $L_{G}^{2}$ respectively;

\item $\langle\psi_{G},U_{FG}\varphi_{F}r \rangle_{G} \, =\, \langle
U_{GF}\psi_{G},\varphi_{F}\rangle_{F}$ for all $\psi_{G}\in L_{G}^{2}$,
$\varphi_{F}\in L_{F}^{2}$.
\end{enumerate}
\end{theorem}

\begin{proof}
(1) To show that the linear operators are isometries:
\begin{align*}
\left\Vert U_{FG}\varphi_{F}\right\Vert _{G}^{2}  &  =\int_{A_{G}}%
|U_{FG}\varphi_{F}|^{2}d\mu_{G}\\
&  =\int_{A_{G}}|\varphi_{F}\circ T_{GF}|^{2}d\mu_{G}\\
&  =\int_{A_{F}}|\varphi_{F}|^{2}d(\mu_{G}\circ T_{FG})\\
&  =\int_{A_{F}}|\varphi_{F}|^{2}d\mu_{F}=\left\Vert \varphi_{F}\right\Vert
_{F}^{2},
\end{align*}
the third equality from the change of variable formula and Lemma~\ref{lem:bi};
the fourth equality from statement (2) of Theorem~\ref{thm:MP}.

(2) From the definition of the induced isometries
\[
(U_{GF} \, U_{FG} (\varphi_{F})) (x) = \varphi_{F}( T_{GF} \, T_{FG} (x) ).
\]
But by Lemma~\ref{lem:bi}, the fractal transformations $T_{GF}$ and $T_{FG}$
are inverses of each other almost everywhere. Therefore the functions $U_{GF}
\, U_{FG} (\varphi_{F})$ and $\varphi_{F}$ are equal for almost all $x\in
A_{F}$.

(3) This is an exercise in change of variables, similar to the proof of (1).
\end{proof}

\begin{example}
[The Cantor function]\label{ex:CF} Consider the two IFS's $F=\{\mathcal{C}%
;\,\frac{1}{3}x,\frac{1}{3}x+\frac{2}{3}\}$ and $G=\{[0,1];\,\frac{1}%
{2}x,\frac{1}{2}x+\frac{1}{2}\}$, the first with attractor equal to the
standard Cantor set $\mathcal{C}$, the second with attractor equal to the unit
interval. In this case the fractal transformation $T_{FG}:\mathcal{C}%
\rightarrow\lbrack0,1]$ is essentially the Cantor function. The Cantor
function is usually defined as a function $f:[0,1]\rightarrow\lbrack0,1]$ so
that if $x$ is expressed in ternary notation as $x=i_{1}\,i_{2}\,\cdots$ where
$i_{k}\in\{0,1,2\}$ for all $k$, then $f(x)=i_{1}^{\prime}\,i_{2}^{\prime
}\,\cdots$ expressed in binary, where $i^{\prime}=0$ if $i\in\{0,1\}$ and
$i^{\prime}=1$ if $i=2.$ The function $T_{F,G}:\mathcal{C}\rightarrow
\lbrack0,1]$ is essentially the same except the domain is $\mathcal{C}$ rather
than $[0,1]$.
\end{example}

Let $F$ and $G$ be IFSs with the same probability vectors and corresponding
invariant measures $\mu_{F}$ and $\mu_{G}$. If $\{e_{n}\}$ is an orthonormal
basis for $L_{F}^{2}$, then by Theorem~\ref{thm:isometry}, the set
$\{\widehat{e}_{n}\}=\{U_{FG}\,e_{n}\}$ is an orthonormal basis for $L_{G}%
^{2}$. In the following example, the two IFSs $F$ and $G$ have the same
attractor $A_{F}=A_{G}=[0,1]$, and the invariant measures are both Lebesque
measure. For example, the Fourier orthonormal basis $\{e^{2\pi inx}%
\}_{n=-\infty}^{\infty}$ of $L^{2}([0,1])$ is transformed under $U_{FG}$ to a
\textquotedblleft fractalized" orthonormal basis of $L^{2}([0,1])$. Therefore,
to any function in $L^{2}([0,1])$ there is a Fourier series and also
corresponding (via $T_{FG}$) a fractal Fourier series. (ii) To prove that
$U_{FG}U_{GF}=I_{F}$ we remove from $A_{F}$ all point that have more than one
address w.r.t. $F$, i.e. those point $x\in A_{F}$ for which $\pi_{F}^{-1}(x)$
is not a singleton and we also remove those points of $A_{F}$ for which
$\pi_{G}^{-1}(T_{FG}(x))$ is not a singleton; this is the set $A_{F}^{G}$
defined earlier; it has full measure, and $T_{GF}T_{FG}|_{A_{F}^{G}}$ is the
identity on $A_{F}^{G}$.

\subsection{Fractal Fourier sine series}

\label{sec:ss}%

\begin{figure}[ptb]%
\centering
\fbox{\includegraphics[
trim=0.000000in 0.000000in 0.047446in 0.000000in,
natheight=3.413400in,
natwidth=3.413400in,
height=1.5342in,
width=3.039in
]%
{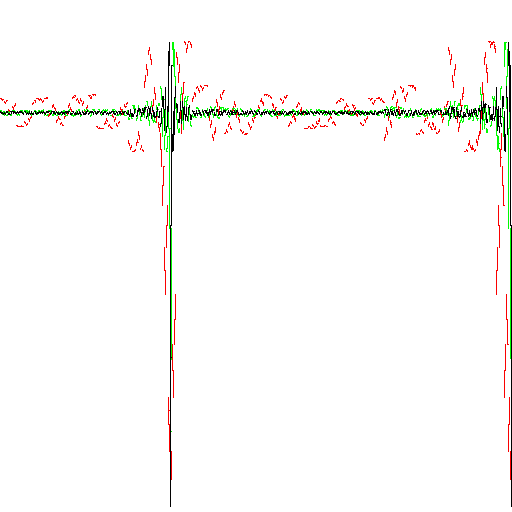}%
}\caption{Fractal sine series approximations to a constant function on the
interval $[0,1].$ The number of terms used here are $10$ (red), $50$ (green)
and $100$ (black). Compare with Figure \ref{fg06}; the r.m.s. errors are the
same as for the approximation to the same constant function using a sine
series with the same number of terms. Notice that the edge effect has been
shifted from $0$ to $1/3$.}%
\label{ffs05}%
\end{figure}
\begin{figure}[ptb]%
\centering
\fbox{\includegraphics[
natheight=3.413400in,
natwidth=3.413400in,
height=1.5342in,
width=3.039in
]%
{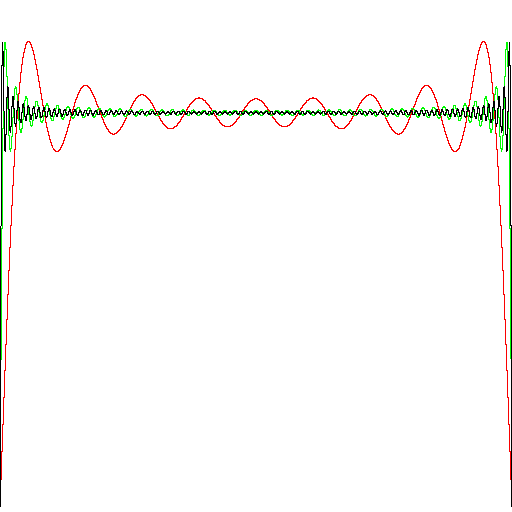}%
}\caption{For comparison with Figure \ref{ffs05}, this shows the Fourier sine
series approximations to a constant function on $[0,1]$ using $k=10$ (red),
$50$ (green) and $100$ (black) significant terms. Note the well-known end
effects at the edges of the interval.}%
\label{fg06}%
\end{figure}

Consider the IFSs $F,G_{1},G_{2}$ of Example~\ref{ex:1} with probabilities
$p_{1}=p_{2}=0.5$. In this case $\mu_{F},\mu_{G_{1}}$ and $\mu_{G_{2}}$ are
all Lebesque measure on $[0,1]$. Consider the orthonormal Fourier sine basis
$\{\sqrt{2}\,e_{n}\}_{n=1}^{\infty}$ for $L^{2}[0,1]$, where $e_{n}=\sin(n\pi
x)$.

For the fractal transformation $T_{FG_{1}},$ the fractally transformed
orthonormal basis for $L^{2}[0,1]$ is \linebreak$\{\sqrt{2}\,\widehat{e}%
_{n}\}_{n=1}^{\infty}$, where
\[
\widehat{e}_{n}(x)=\sin(n\pi T_{G_{1}F}(x)),
\]
for all $n\in\mathbb{N}$. Figure \ref{sinallpix} illustrates ${\widehat{e}%
_{i}},\,i=1,2,3$, in colors black, red, and green, respectively. For
comparison, Figure \ref{sines} illustrates the corresponding sine functions
$\sin(n\pi x)$ for $n=1,2,3$.%

\begin{figure}[ptb]%
\centering
\fbox{\includegraphics[
natheight=3.413400in,
natwidth=3.413400in,
height=1.5238in,
width=3.039in
]%
{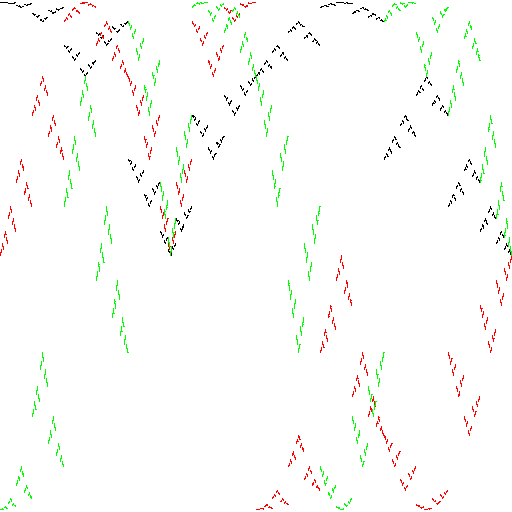}%
}\caption{See text. The first three eigenfunctions of the elementary fractal
transformed Laplacian on [0,1]; equivalently, the functions f-sin(n,x) for n=1
(black), 2 (red), 3 (green). The viewing window is [0,1]$\times$[-1,1].}%
\label{sinallpix}%
\end{figure}
%

\begin{figure}[ptb]%
\centering
\fbox{\includegraphics[
natheight=3.413400in,
natwidth=3.413400in,
height=1.5342in,
width=3.039in
]%
{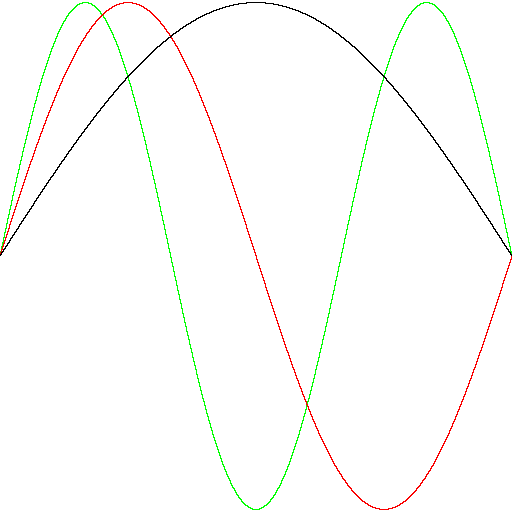}%
}\caption{This illustrates the sine functions $sin(n\pi x)$ for n=1,2,3 for
comparison with the fractal sine function shown in Figure \ref{sinallpix}.}%
\label{sines}%
\end{figure}

\begin{example}
[Constant function]Figure \ref{ffs05} illustrates three fractal Fourier sine
series approximations to a constant function on the interval $[0,1]$, while
Figure \ref{fg06} illustrates the standard sine series Fourier approximation
using the same numbers of terms. The respective Fourier series are
\[
\sum_{n=1}^{\infty}\frac{\widehat{e}_{2n-1}(x)}{2n-1}\qquad\text{and}%
\qquad\sum_{n=1}^{\infty}\frac{e_{2n-1}\,(x)}{2n-1}.
\]
The calculation, in the first case, of the Fourier coefficients, uses the
change of variables formula, the fact from Example~\ref{ex:1} that $\mu_{F}$
and $\mu_{G_{!}}$ are Lebesque measure, and statement 2 of
Theorem~\ref{thm:MP}. The mean square errors are the same when using the same
number of terms.
\end{example}

\begin{example}
[Step function]Next consider Fourier approximants to a step function. The
fractal transformation $T_{FG_{2}}$ has fractal sine functions defined by
\[
\widetilde{e}_{n} := \sin(n\pi T_{G_{2}F}(x))
\]
for all $n\in\mathbb{N}$. Figures \ref{step0}, \ref{step1}, and \ref{step2}
illustrate the Fourier approximations for $100$ (green) and $500$ (black)
terms, where the orthogonal bases functions are $e_{n}, \; \widehat{e}_{n} $
and $\widetilde{e}_{n},$ respectively. The respective Fourier series are
\[
\frac{2}{\pi}\, \sum_{n=1}^{\infty} \frac{1-\cos(n\pi/2)}{n} \, f_{n}(x),
\]
where $f_{n}$ is $e_{n}, \; \widehat{e}_{n} $ and $\widetilde{e}_{n},$
respectively. The point to notice is that the jump in the step function at
$x=0.5$ is cleanly approximated in both the fractal series, in contrast to the
well-known edge effect (Gibbs phenomenon) in the classical case. The price
that is paid is that the fractal approximants have greater pointwise errors at
some other values of $x$ in $[0,1]$. The analysis of where this occurs and
proof that the mean square error is the same for all three schemes, is omitted here.%

\begin{figure}[ptb]%
\centering
\fbox{\includegraphics[
natheight=3.413400in,
natwidth=3.413400in,
height=2.0358in,
width=2.0358in
]%
{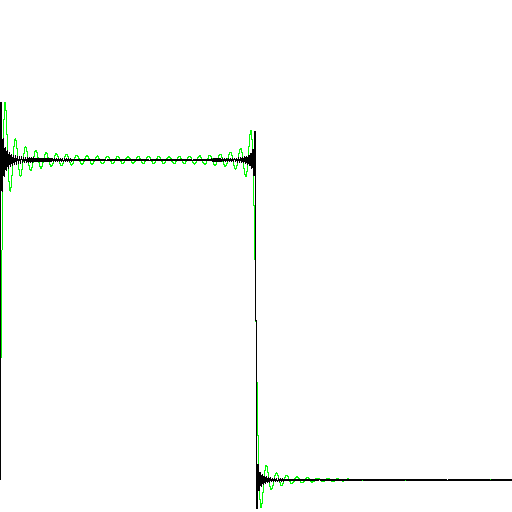}%
}\caption{Sum of the first 100 (green) and 500 (black) terms in the Fourier
sine series for a step function. The viewing window is [0,1]$\times
$[-0.1,1.5]. Compare with Figures \ref{step1} and \ref{step2}. }%
\label{step0}%
\end{figure}
%

\begin{figure}[ptb]%
\centering
\fbox{\includegraphics[
natheight=3.413400in,
natwidth=3.413400in,
height=2.0358in,
width=2.0358in
]%
{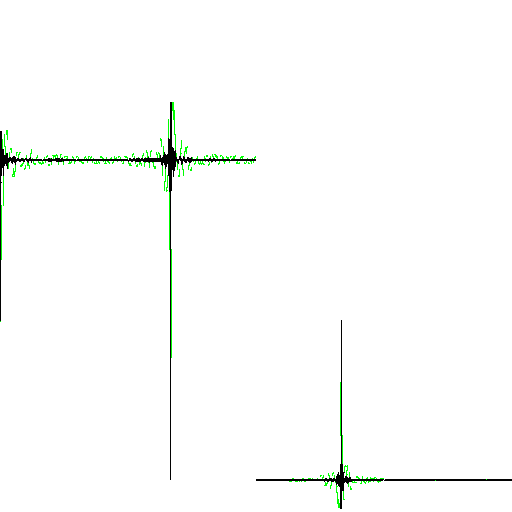}%
}\caption{Sum of the first 100 (green) and 500 (black) terms in a fractal
Fourier sine series (using f-sin(n,x) functions) for a step function. Compare
with Figures \ref{step0} and \ref{step2}.}%
\label{step1}%
\end{figure}
%

\begin{figure}[ptb]%
\centering
\fbox{\includegraphics[
natheight=3.413400in,
natwidth=3.413400in,
height=2.0358in,
width=2.0358in
]%
{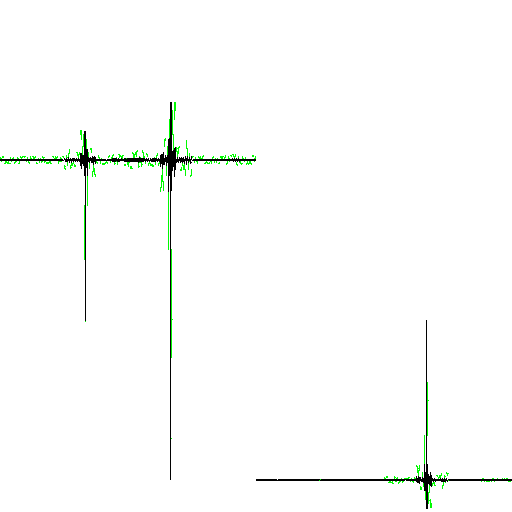}%
}\caption{Sum of the first 100 (green) and 500 (black) terms in a fractal
Fourier sine series (using f2-sin(n,x) functions) for a step function. Compare
with Figures \ref{step0} and \ref{step1}.}%
\label{step2}%
\end{figure}

\end{example}

\begin{example}
[Tent function]\label{ex:tent}In Figure \ref{tent2} partial sums of the
Fourier sine series and their fractal counterparts are compared, for the tent
function $f(x)=\min\{x,1-x\}$ on the unit interval. The Fourier series with
orthogonal functions $e_{n}$ is compared with the Fourier series with fractal
orthogonal functions $\widetilde{e}_{n}$, using $3$ (red), $5$ (green), $7$
(blue), $20$ (black) terms. The Fourier series are (up to a normalization
constant)%
\[
\sum_{n=1}^{k}\frac{2\sin(\pi n/2)-\sin(\pi n)}{n^{2}}\,e_{n}(x)\ \qquad\text{
and }\qquad\sum_{n=1}^{k}\frac{2\sin(\pi n/2)-\sin(\pi n)}{n^{2}}%
\,\widetilde{e}_{n}(x).
\]

\end{example}%

\begin{figure}[ptb]%
\centering
\fbox{\includegraphics[
natheight=3.413400in,
natwidth=3.413400in,
height=2.0358in,
width=2.0358in
]%
{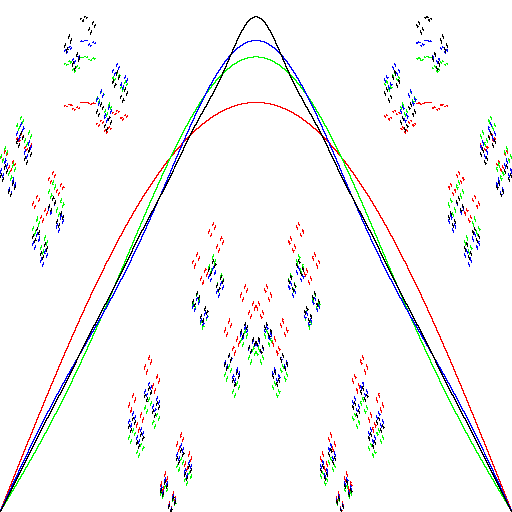}%
}\caption{See Example \ref{ex:tent}. Fourier sine series approximants to a
tent function and fractal counterparts.}%
\label{tent2}%
\end{figure}

\begin{example}
[Function with a dense set of discontinuities]\label{ex:dd} Consider the
following approximation of a function with a dense set of discontinuities. For
$i=1,2$, let $\psi\in L^{2}[0,1]$ be defined by $\psi(x)=x$ for all
$x\in\lbrack0,1]$. Then $\phi_{i}=U_{FG_{i}}\psi,\,i=1,2,$ is given by
$\phi_{i}(x)=(U_{FG_{i}}\psi)(x)=\psi(T_{G_{i}F}(x))=T_{G_{i}F}(x)$, which has
a dense set of discontinuities. It follows, by a short calculaltion using
statement 2 of Theorem~\ref{thm:MP}, that the coefficients in the $\widehat
{e}_{n}$ and $\widetilde{e}_{n}$ Fourier series expansion of $\phi_{i}$ are
the same as the coefficients in the $e_{n}$ expansion for $\psi$. Therefore
the fractal version Fourier series expansions for $\phi_{i},\,i=1,2,$ are
\[
\frac{2}{\pi}\sum_{n=1}^{\infty}\frac{-\cos(\pi n)}{n}\;\widehat{e}%
_{n}(x),\qquad\text{and}\qquad\frac{2}{\pi}\sum_{n=1}^{k}\frac{-\cos(\pi
n)}{n}\,\widetilde{e}_{n}(x),
\]
respectively. Sums with $10,30,$ and $100$ terms are shown in red, green, and
blue, respectively, in Figure \ref{fx} for $\phi_{1}$, and for $\phi_{2}$ in
Figure \ref{n500b} using the first 1000 terms of the series.%

\begin{figure}[ptb]%
\centering
\fbox{\includegraphics[
natheight=3.413400in,
natwidth=3.413400in,
height=2.0358in,
width=2.0358in
]%
{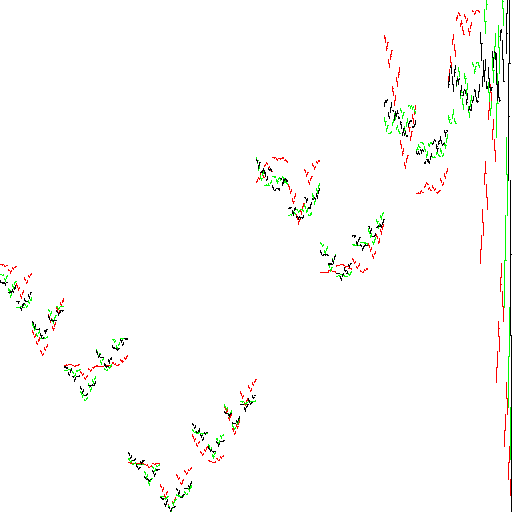}%
}\caption{See Example  \ref{ex:dd}. Compare with Figure \ref{tfg02}. The
approximants converge to $T_{G_{1}F}(x)$ in $\mathcal{L}^{2}[0,1]$ as the
number of terms in series sum approaches infinity.}%
\label{fx}%
\end{figure}
%

\begin{figure}[ptb]%
\centering
\fbox{\includegraphics[
natheight=14.221800in,
natwidth=22.583700in,
height=2.0349in,
width=2.0358in
]%
{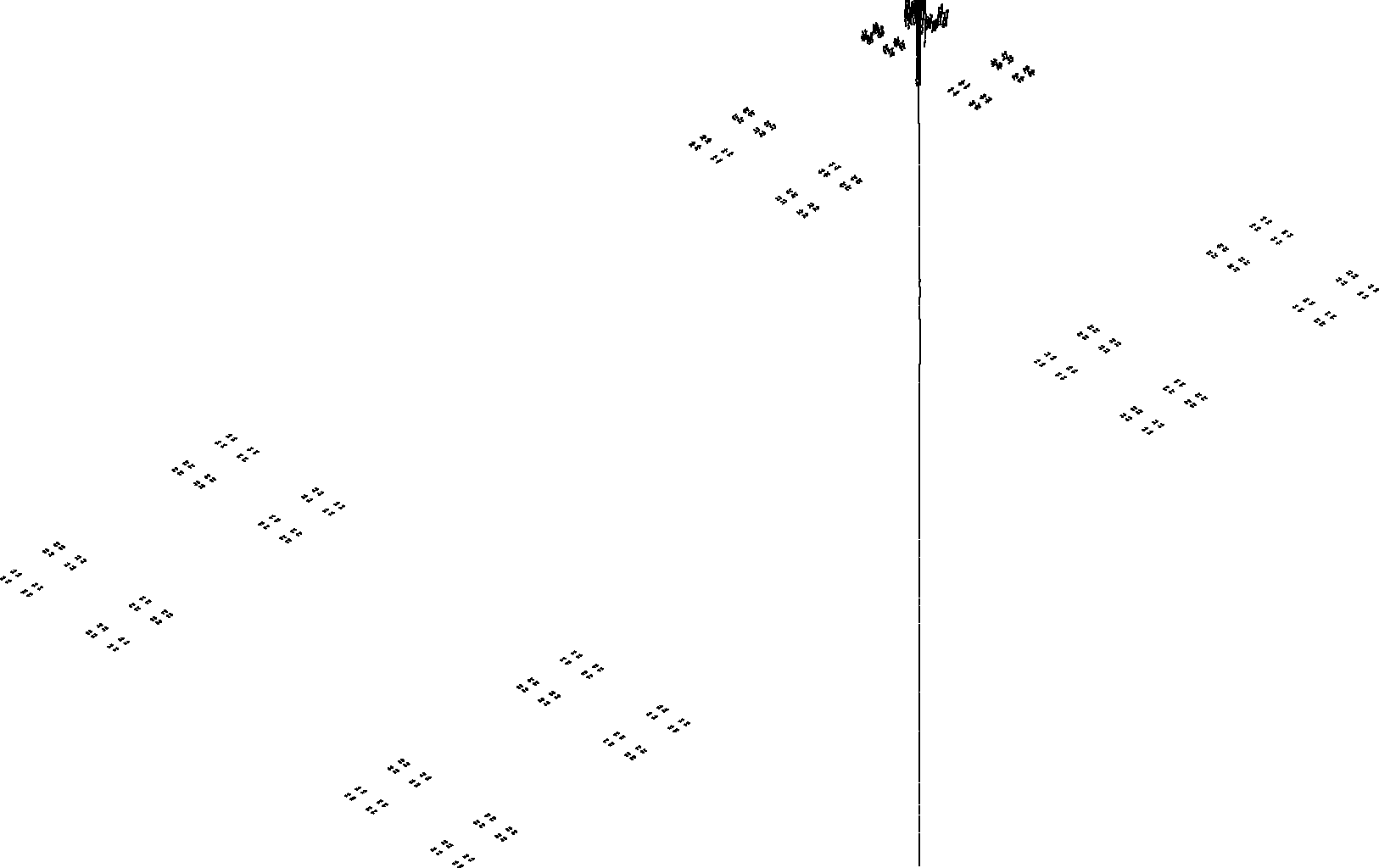}%
}\caption{See Example \ref{ex:dd}. This illustrates the sum of the first
thousand terms of a fractal sine series for $T_{FG_{2}}(x)$ on $[0,1]$.
Compare with Figure \ref{2ndfsin01}.}%
\label{n500b}%
\end{figure}

\end{example}

\subsection{\label{sec:legendre} Legendre polynomials. }

The Legendre polynomials are the result of applying  Gram-Schmidt
orthogonalization $\{1,x,x^{2},\dots\}$, with respect to Lebesgue measure on
$[-1,1]$. Denote the Legendre polynomials shifted to the interval $[0,1]$ by
$\{P_{n}(x)\}_{n=0}^{\infty}$. They form a complete orthogonal basis for
$L^{2}[0,1]$, where the inner product is
\[
\langle\psi,\varphi\rangle\,=\int\limits_{0}^{1}\overline{\psi(x)}%
\varphi(x)dx\text{.}%
\]

In this case each of the unitary transformations $U_{FG}$ associated with
Example \ref{ex:1} maps $L^{2}[0,1]$ to itself, and we obtain the
\textquotedblleft fractal Legendre polynomials"%
\[
P_{n}^{FG}(x)=P_{n}(T_{GF}(x))\text{.}%
\]
With $F,G_{1},G_{2}$ as previously defined in Example~\ref{ex:1}, Figures
\ref{legendref} and \ref{legendref2} illustrate the Legendre polynomials and
their fractal counterparts. Figure \ref{legendref} shows the fractal Legendre
polynomials $P_{n}^{FG_{1}}(x)$ and Figures \ref{legendref2} shows the fractal
Legendre polynomials $P_{n}^{FG_{2}}(x)$.%

\begin{figure}[ptb]%
\centering
\fbox{\includegraphics[
natheight=3.413400in,
natwidth=3.413400in,
height=2.0358in,
width=2.0358in
]%
{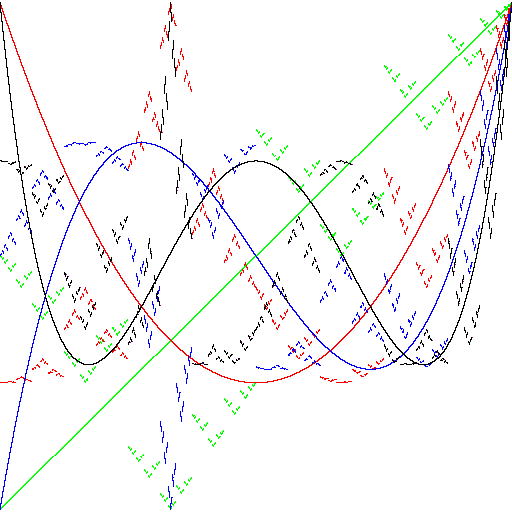}%
}\caption{Legendre polynomials and their fractal counterparts corresponding to
$T_{FG_{1}}$. Both sets of functions form orthogonal basis sets with respect
to Lebesgue measure on the interval $[-1,1]$. See also Figure \ref{Legendref2}%
.}%
\label{legendref}%
\end{figure}
%

\begin{figure}[ptb]%
\centering
\fbox{\includegraphics[
natheight=3.413400in,
natwidth=3.413400in,
height=2.0358in,
width=2.0358in
]%
{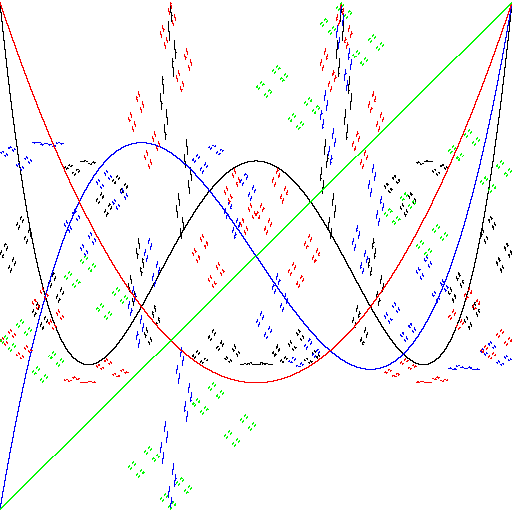}%
}\caption{Legendre polynomials and their fractal counterparts corresponding to
$T_{FG_{2}}$. See also Figure \ref{legendref}.}%
\label{legendref2}%
\end{figure}

\subsection{ \label{ex2b(ctd)HaarX} The action of the unitary operator on Haar
wavelets.}

With $F, G_{2}$ and $T=T_{FG_{2}}:\mathbb{[}0,1]\rightarrow\lbrack0,1]$ as
previously defined, let $U = U_{FG_{2}} :L^{2}[0,1]\rightarrow L^{2}[0,1]$ be
the associated (self-adjoint) unitary transformation. Let $I_{\emptyset
}=[0,1]$ and $H_{\emptyset}:\mathbb{R}\rightarrow\mathbb{R}$ be the Haar
mother wavelet defined by%
\[
H_{\emptyset}(x)=\left\{
\begin{array}
[c]{c}%
+1\text{ if }x\in\lbrack0,0.5)\text{,}\\
-1\text{ if }x\in\lbrack0.5,1)\text{,}\\
0\text{ otherwise.}%
\end{array}
\right.
\]
For $\sigma\in\{0,1\}^{k}, \, k\in\mathbb{N}$, write $\sigma=\sigma_{1}%
\sigma_{2}...\sigma_{k}$ and $\left\vert \sigma\right\vert =k$. If $\left\vert
\sigma\right\vert =0$ then $\sigma=\emptyset$, the empty string. Also let
$I_{\sigma}=h_{\sigma_{1}}\circ h_{\sigma_{2}}\circ...\circ h_{\sigma_{k}%
}(I_{\emptyset}),$ where $h_{0}=f_{1}$ and $h_{1}=f_{2}$, and let $A_{\sigma
}:\mathbb{R\rightarrow R}$ be the unique affine map such that $A_{\sigma
}(I_{\emptyset})=I_{\sigma}$. With this notation, the \textbf{standard Haar
basis}, a complete orthonormal basis for $L^{2}[0,1]$, is
\[
\{H_{\sigma}:\sigma\in\{0,1\}^{k}, k\in\mathbb{N\}\cup\{}H_{\emptyset
}(x)\}\cup\{\mathbf{1}\},
\]
where $\mathbf{1}$ is the characteristic function of $[0,1)$ and $H_{\sigma
}:[0,1)\rightarrow\mathbb{R}$ is defined by
\[
H_{\sigma}(x)=2^{\left\vert \sigma\right\vert /2}H_{_{\emptyset}}(A_{\sigma
}^{-1}(x))\text{.}%
\]

There is an interesting action of $U=U_{FG_{2}}$ on Haar wavelets. The
operator $U$ permutes pairs of Haar wavelets at each level and flips signs of
those at odd levels, as follows. By calculation, for $\sigma\in\cup
_{k\in\mathbb{N}}\{0,1\}^{k}$,
\[
UH_{\sigma}=(-1)^{\left\vert \sigma\right\vert }H_{\sigma^{\prime}}%
\]
where $\left\vert \sigma\right\vert =\left\vert \sigma^{\prime}\right\vert $
and $\sigma_{l}^{\prime}=(-1)^{l+1}\sigma_{l}+(1+(-1)^{k})/2$ for all
$l=1,2,...,\left\vert \sigma^{\prime}\right\vert $, $UH_{\emptyset
}=H_{\emptyset}$, and $U1=1$. It follows that if $f\in L^{2}[0,1]$ is of the
special form%
\[
f=a_{\emptyset}H_{\emptyset}+\sum_{\sigma\in\cup_{k\in\mathbb{N}}\{0,1\}^{2k}%
}c_{\sigma}(H_{\sigma}+H_{\sigma^{\prime}})\text{,}%
\]
then $Uf=f$ and $f\circ T=f$.\ Such signals are invariant under $U$. It also
follows that if $P$ is the projection operator that maps $L^{2}[0,1]$ onto the
span of all Haar wavelets down to a fixed depth, then $U^{-1}PU=P$.

\subsection{Unitary transformations from the Hilbert mapping and its
inverse\label{ex3(ctd)}}

This continues Example~\ref{ex:Hilbert}, where the fractal transformations
$h:=T_{FG}$ and $h^{-1}:=T_{GF}$ are the Hilbert mapping and its inverse, both
of which both preserve Lebesgue measure and are mappings between one and two
dimensions. The unitary transformations $U_{FG}\,:\,L^{2}([0,1])\rightarrow
L^{2}([0,1]^{2})$ and $U_{GF}\,:\,L^{2}([0,1]^{2})\rightarrow L^{2}([0,1])$
are given by
\[
U_{FG}(f)=f\circ h^{-1},\qquad\qquad U_{GF}(f)=f\circ h.
\]
A \textit{picture} can be considered as a function $f:[0,1]^{2}\rightarrow
\mathbb{R}^{3}$, where the image of a point $x$ in ${\mathbb{R}}^{3}$ gives
the RGB colours. The top image of Figure \ref{lenawiline} is a picture of the
graph of such a function $f:[0,1]^{2}\rightarrow\mathbb{R}^{3}$. The bottom
image is the function (picture) $U_{GF}f=f\circ h$ transformed by the unitary operator.%

\begin{figure}[ptb]%
\centering
\includegraphics[
natheight=3.560000in,
natwidth=5.119900in,
height=1.9545in,
width=3.6123in
]%
{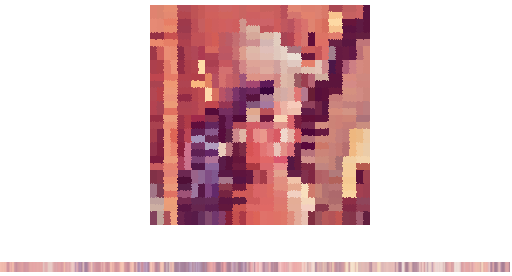}%
\caption{See Section \ref{ex3(ctd)}. }%
\label{lenawiline}%
\end{figure}

The Hilbert map $h:[0,1]\rightarrow\lbrack0,1]^{2}$ is continuous, one
consequence of which is that, if $f:[0,1]^{2}\rightarrow\mathbb{R}^{3}$ is
continuous, then so is the pull-back $U_{GF}(f)=f\circ h:[0,1]\rightarrow
\mathbb{R}^{3}$. To illustrate, any orthonormal basis w.r.t. Lebesgue measure
on $[0,1]$ is mapped, via the unitary operator $U_{FG}$, to an orthonormal
basis w.r.t. Lebesgue measure on $[0,1]^{2}$, and conversely. Because the
Hilbert mapping is continuous, an orthonormal basis of continuous functions
$\{\psi_{n}:[0,1]^{2}\rightarrow$ $\mathbb{R\}}$ is transformed by $U_{GF}$ to
an orthonormal basis of continuous functions $\{\psi_{n}\circ
h:[0,1]\rightarrow$ $\mathbb{R\}}$. In the other direction, the image of an
orthonormal basis consisting of continuous functions on $[0.1]$ may not
comprise continuous functions on $[0,1]^{2}$. Figures \ref{sinall} and
\ref{ctssinhilbert} illustrate this.

In Figure \ref{both}, the right image represents the graph of $f:[0,1]^{2}%
\rightarrow\lbrack-1,1]$ defined by $f(x,y)=\sin(\pi x)\sin(\pi y)$. The left
image represents the graph of $g:[0,1]^{2}\rightarrow\lbrack-1,1]$ defined by
the continuous function $g(x,y)=U_{GF}(f)=f\circ h(x)$ where
$h:[0,1]\rightarrow\lbrack0,1]^{2}$ is the Hilbert function. The set of
functions in the orthogonal basis $\left\{  \sin(n\pi x)\sin(m\pi
y):n,m\in\mathbb{N}\right\}  $ for $L^{2}([0,1]^{2})$ (w.r.t.Lebesque
two-dimensional measure) is fractally transformed via the Hilbert mapping to
an othogonal basis for $L^{2}[0,1]$ (w.r.t.Lebesgue one-dimensional measure).
In contrast to the situation in Section~\ref{sec:ss}, these "fractal sine
functions" are continuous.%

\begin{figure}[ptb]%
\centering
\includegraphics[
natheight=2.860100in,
natwidth=5.119900in,
height=2.0299in,
width=3.6123in
]%
{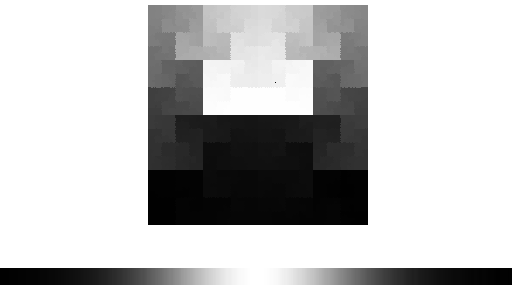}%
\caption{ The bottom band shows the graph of $\sin(\pi x)$ with function
values represented by shades of grey. The top band shows the graph of
$h(sin(\pi x))$, where $h$ is the Hilbert function.}%
\label{sinall}%
\end{figure}
%

\begin{figure}[ptb]%
\centering
\includegraphics[
natheight=1.766700in,
natwidth=3.276900in,
height=2.0324in,
width=3.6123in
]%
{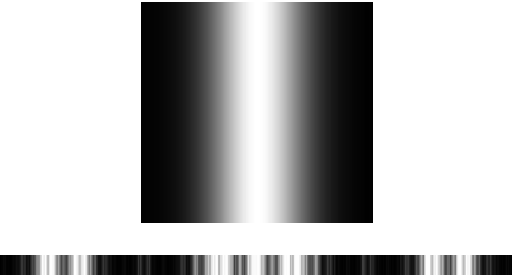}%
\caption{The top image illustrates the graph of $f(x,y)=\sin(\pi x)$ for
$x,y\in\lbrack0,1]^{2}$. The band at the bottom illustrates the graph of the
pull-back $f\circ h:[0,1]\rightarrow\lbrack-1,1]$, which is continuous, in
contrast to the situations in Figures \ref{lenawiline}.}%
\label{ctssinhilbert}%
\end{figure}
%

\begin{figure}[ptb]%
\centering
\includegraphics[
natheight=3.276900in,
natwidth=6.617600in,
height=1.9363in,
width=3.8821in
]%
{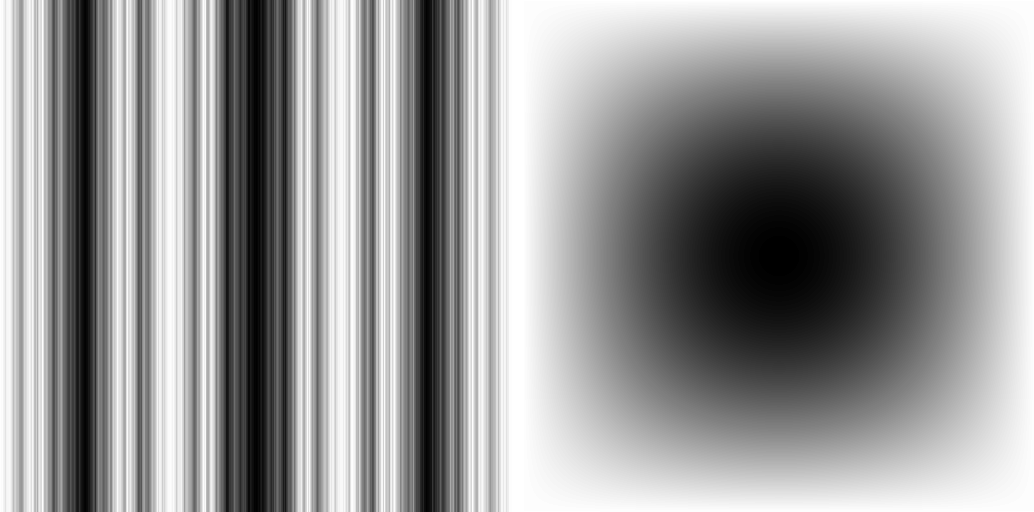}%
\caption{The right image represents the graph of $f:[0,1]^{2}\rightarrow
\lbrack-1,1]$ defined by $f(x,y)=\sin(\pi x)\sin(\pi y)$. The left image
represents the graph of $g:[0,1]^{2}\rightarrow\lbrack-1,1]$ defined by the
continuous function $g(x,y)=U_{GF}(f)=f\circ h(x)$ where $h:[0,1]\rightarrow
\lbrack0,1]^{2}$ is the Hilbert function. }%
\label{both}%
\end{figure}

\section{Fractal Transformation of a Linear Operator}

\label{sec:LO}

Let $F$ and $G$ be IFSs with the same number of functions. Using the same
notation as in the previous section, if $W_{F} \, : \,L^{2}_{F} \rightarrow
L^{2}_{F}$ is a linear operator, then the fractally transformed linear
operator $W_{G} : \,L^{2}_{G} \rightarrow L^{2}_{G}$ defined by
\[
W_{G} =U_{FG} \circ W_{F} \circ U_{GF}%
\]
is also a linear operator. If $W_{F}$ is a bounded, self-adjoint linear
ooperator with spectral representation
\[
W_{F}=\int_{-\infty}^{+\infty}\lambda dP_{\lambda}^{F}\text{,}%
\]
where $P_{\lambda}^{F}$ is an increasing family of projections on $L_{F}^{2}$,
then
\[
W_{G}=\int_{-\infty}^{+\infty}\lambda dP_{\lambda}^{G}%
\]
where $P_{\lambda}^{G}=U_{FG} \circ P_{\lambda}^{F}\circ U_{GF}$. In
particular, $W_{F}$ and $W_{G}$ have the same spectrum.

\subsection{Differentiable functions}

\label{sec:diff}

\begin{definition}
Let $F$ and $G$ be IFSs with $F$ and $G$ non-overlapping, and $T_{FG}$ the
fractal transformation from $A_{F}$ to $A_{G}$. Assume that the attractor
$A_{F}$ of $F$ is the interval $[0,1]$, and denote the $k$ times continuously
differentiable functions $f : A_{F} = [0,1] \rightarrow{\mathbb{R}}$ by
$C_{F}^{k} $. The set
\[
C_{G}^{k} = \{ U_{FG} f \, : \, f \in C_{F}^{k} \}
\]
will be called $k$ \textbf{times continuously differentiable fractal
functions}. If the $k^{th}$ derivative of $f \in C_{F}^{k}$ is denoted
$D_{F}^{k} \, f$, where $D_{F}^{k}$ is the differential operator, then
\[
D_{G}^{k} \, g := (U_{FG}\circ D_{F}^{k} \circ U_{GF})\, g
\]
will be referred to as the $k^{th}$ \textbf{fractal derivative of} $g \in
C_{G}^{k}$.
\end{definition}

Note that analogous definitions can be made when $A_{F}$ is a subset of
${\mathbb{R}}^{n}$ with nonempty connected interior, for example a square or
filled triangle in the plane. In that case, we have partial derivatives.

To obtain an intuitive interpretation of the fractal derivative, consider the
case where the attractor of $F$ (with probability vector $p$) is $[0,1]$ as
above and $F$ has the property that $\pi_{F}$ is an increasing function from
the code space to $[0,1]$ with respect to the lexicographic order on the code
space. Assume, similarly, that $G$ (with the same probability vector $p$) has
the property that there is a linear order $\preceq$ on $A_{G}$ such that $\pi$
is increasing with respect to this order on $A_{G}$ and the lexicographic
order on the code space. Assume further that $T_{FG}$ is a fractal
homeomorphism. Note that all the above assumptions hold in
Examples~\ref{ex:CF} of the Cantor set and Example~\ref{ex:1} of the Koch curve.

For $y_{1}, y_{2} \in A_{G}$ we use the following notation for the interval:
$[y_{1},y_{2}] = \{ y \,:\, y_{1} \preceq y \preceq y_{2}\}$. Under these
assumptions, and with Lebesque measure $\mu$ as the invariant measure of $F$
and $\mu_{G}$ the invariant measure of $G$, define the \textbf{fractal
difference} betwen a pair of points in $A_{G}$ by
\[
y_{1} - y_{2} =
\begin{cases}
\mu_{G}([y_{2}, y_{1}]) \qquad\text{ if } \; y_{1} \geq y_{2},\\
-\mu_{G}([y_{1},y_{2}]) \quad\text{ if } \; y_{1} < y_{2}.
\end{cases}
\]

\begin{theorem}
With notation as above, if $g : A_{G} \rightarrow{\mathbb{R}}$ is a
differentiable fractal function, then
\[
D_{G} \, g \, (y_{0}) = \lim_{y \rightarrow y_{0}} \frac{g(y ) - g(y_{0} )}{ y
- y_{0}}.
\]

\end{theorem}

\begin{proof}
If $g : A_{G} \rightarrow{\mathbb{R}}$ is a differentiable fractal function,
then there is an $f : [0,1] \rightarrow{\mathbb{R}}$ such that $g = U_{FG} \,
f$. Now
\[
\begin{aligned} D_G \, g(y_0) &= (U_{FG} \circ \frac{d}{dx} \circ U_{GF} \, g ) \,(y_0) = ((U_{FG} \circ \frac{d}{dx} \circ U_{GF}) \,( U_{FG}\,f ) ) \, (y_0) \\ & =( (U_{FG} \circ \frac{d}{dx} ) \, f) (y_0) = (U_{FG} \circ f') \,(y) = f'( T_{GF}\, y_0) = \lim_{x\rightarrow T_{GF} \, y_0} \frac{f(x) - f(T_{GF}\, y_0 )}{x - T_{GF}\, y_0}. \end{aligned}
\]
Given $x \in[0,1]$, there is a unique $y_{x} \in A_{G}$ such that $T_{GF} \,
y_{x} = x$. Moreover, since $T_{GF}$ is continuous, as $y \rightarrow y_{0}$
we have $T_{GF} \, y \rightarrow T_{GF} \, y_{0}$, i.e., $x \rightarrow T_{GF}
\, y_{0}$. Therefore
\[
D_{G} \, g(y_{0}) = \lim_{y \rightarrow y_{0}} \frac{f(T_{GF} \,y ) -
f(T_{GF}\, y_{0} )}{ T_{GF}\, y - T_{GF}\, y_{0}} = \lim_{y \rightarrow y_{0}}
\frac{g(y ) - g(y_{0} )}{ T_{GF}\, y - T_{GF}\, y_{0}}.
\]
By Lemma~\ref{lem:ZF}, if $\nu$ is the invariant measure on code space with
probability vector $p$, then (assume $y \succeq y_{0}$ without loss of
generality)
\[
\begin{aligned} \mu ( [T_{GF}\, y_0, T_{GF} \, y] ) &= \nu ( \pi_F^{-1} ([ T_{GF}\, y_0, T_{GF}\, y]) = \nu ( [ \pi_F^{-1} \, T_{GF}\, y_0, \pi_F^{-1}\, T_{GF}\, y] ) \\ &= \nu ([\tau_G\, y_0, \tau_G \, y]) = \nu ( [ \pi_G^{-1} \, y_0, \pi_G^{-1} \, y] ) = \nu ( \pi_G^{-1} ( [ y_0, y] ) ) \\ &= \mu_G ( [ y_0, y] ) = y - y_0. \end{aligned}
\]
Therefore
\[
D_{G} \, g \, (y_{0}) = \lim_{y \rightarrow y_{0}} \frac{g(y ) - g(y_{0} )}{ y
- y_{0}}.
\]

\end{proof}

\begin{example}
[Derivative of the Cantor function]Consider the two IFS's $F = \{ [0,1]; \,
\frac12 x , \frac12 x + \frac12\}$ and $G = \{ \mathcal{C}; \, \frac13 x ,
\frac13 x + \frac23\}$ of Example~\ref{ex:CF}. Let $T_{FG} : [0,1]
\rightarrow\mathcal{C}$ be the fractal transformation from the unit interval
to the Cantor set. If $f : [0,1] \rightarrow{\mathbb{R}}$ is the function
$f(x) = x$, for example, then $g := U_{FG}\, f = T_{GF}$ is exactly the Cantor
function described in Example~\ref{ex:CF}. The fractal derivative of this
Cantor function $g$ is
\[
(D_{G} \, g) (y) = (U_{FG}\circ D_{F} \circ U_{GF}\, g) (y) = (U_{FG}\circ
D_{F} \circ U_{GF} \circ U_{FG} \, f) (y) = (U_{FG}\circ\mathbf{1}_{F}) (y) =
\mathbf{1}_{G} (y) = 1
\]
for all $y \in\mathcal{C}$, where $\mathbf{1}_{F}$ and $\mathbf{1}_{G}$ are
the constant $1$ functions on $[0,1]$ and $\mathcal{C}$, respectively.
Therefore the Cantor function has constant a.e. fractal derivative $1$.
\end{example}

Since the fractal transformation $T_{FG}$ induces transformations on the set
of points of $A_{F}$, on the set $L_{2}(F)$ of functions on $A_{F}$, and on
the set of linear operators on $L^{2}(F)$, any differential equation on
$A_{F}$ can be transformed into a differential equation on $A_{G}$.

\begin{example}
[Differential equation on the Koch curve]Consider the fractal transformation
of Example~\ref{ex:1}, from the unit interval to the Koch curve. The simple
initial value ODE
\[
\frac{dy}{dx}=y,\quad y(0)=1
\]
on the interval $[0,1]$ with solution $y=e^{x}$ transforms to the fractal ODE
\[
(U_{FG}\circ\frac{d}{dx}\circ U_{GF})\,\widehat{y}=\widehat{y},\qquad
\widehat{y}(T_{FG}0)=1
\]
on the Koch curve. The fractal solution to this ODE is the function
$g:=U_{FG}(\exp)$, i.e., $g(x)=e^{T_{GF}(x)}$.
\end{example}

\section{Fractal Flows}

\label{sec:FF}

Let $(X, \mu)$ be a metric space with Borel measure $\mu$, and let $f :
X\rightarrow X$ be invertible almost everywhere, i.e. if there is a function
$f^{-1} : X \rightarrow X$ such that $f\circ f^{-1}(x) = f^{-1}\circ f(x) - x
$ for all $x$ in a set of measure $1$. Let ${\mathcal{M}}(X)$ be the set of
Borel measures on $X$. Slightly abusing notation, we use the same symbol
$f^{\#}$ for the following induced actions on $L^{2}(X)$ and ${\mathcal{M}%
}(X)$, respectively:
\[
\begin{aligned} f^{\#}(\phi) &= \phi\circ f^{-1} \quad \text{for} \; \phi \in L^2(X) \\ f^{\#}(\mu) &= \mu \circ f^{-1} \quad \text{for} \; \mu \in {\mathcal M}(X). \end{aligned}
\]

Let $F$ be an IFS on the space $X$, $G$ an IFS on the space $Y$, and $T_{FG}
\, : \, X \rightarrow Y$ a fractal transformation. Let $\mu_{F}$ and $\mu_{G}$
be the corresponding invariant measures with respect to the same probability
vector. If $f : X\rightarrow X$ is invertible a.e., then define induced
actions on $Y, L^{2}(Y),$ and ${\mathcal{M}}(Y)$ as follows. Again we use the
same notation ${\widehat f}^{\#}$ for the induced actions, where $y \in Y,
\phi\in L^{2}(Y)$, and $\mu\in{\mathcal{M}}(Y)$:
\[
\begin{aligned} g (y) := &{\widehat f}^{\#}(y) = T_{FG} \circ f \circ T_{GF}(y) \\ &{\widehat f}^{\#}(\phi) = g^{\#} (\phi) = U_{FG} \circ f^{\#} \circ U_{GF}(\phi) \\ &{\widehat f}^{\#} (\mu) = g^{\#} (\mu). \end{aligned}
\]
Note that, if $f$ is measure preserving on $X$, then by Theorem~\ref{thm:MP}
the induced function $g$ is measure preserving on $Y$.

By a flow on a space $X$ is meant a mapping $f \, : \, X \times{\mathbb{R}}
\rightarrow X$, with notation $f_{t}(x)$ often used instead of $f(x,t)$, such
that
\[
\begin{aligned} &f_0(x) = x \\ & f_s (f_t (x)) = f_{s+t} (x) \end{aligned}
\]
for all $x\in X$ and all $s,t \in{\mathbb{R}}$. Applying the induced actions
defined above to each function $f_{t}, \, t \in{\mathbb{R}}$, motivates the
following notion of fractal flows. Note that there are fractal flows on the
metric space $Y$, and the space of square integrable functions $L^{2}(Y)$ and
on the space of measures ${\mathcal{M}}(Y)$.

\begin{definition}
\label{def:flow} A flow $f_{t}$ on $X$ induces flows $(f_{t})^{\#}$ on
$L^{2}(X)$ and ${\mathcal{M}}(X)$, and, given a fractal transformation
$T_{FG}$, the flow $f_{t}$ induces \textbf{fractal flows} $(\widehat{ f_{t}%
})^{\#}$ on $Y, L^{2}(Y),$ and ${\mathcal{M}}(Y)$. Since, for a flow,
$f_{t}^{-1} = f_{-t}$, the explicit formulas for the flows are
\[
\begin{aligned} g_t (y) := &{\widehat f_t}^{\#}(y) = T_{FG} \circ f_t \circ T_{GF}(y) \\ &{\widehat f_t}^{\#}(\phi) = g_t^{\#} (\phi) = U_{FG} \circ f_ty^{\#} \circ U_{GF}(\phi) \\ &{\widehat f}_t^{\#} (\mu) = g_t^{\#} (\mu). \end{aligned}
\]

\end{definition}

If $f_{t}$ is a continuous, measure preserving flow on $(X, \mu)$, then it is
readily checked that the flow $f_{t}^{\#} : L^{2}(X) \rightarrow L^{2}(X)$ is
unitary, and hence provides a strongly continuous one parameter unitary group.
By Stone's theorem \cite{stonesthm} there is a unique self-adjoint operator
$L$ such that
\[
f_{t}^{\#}=e^{itL},
\]
where $iL$ is referred to as the \textit{infinitesimal generator}. Moreover,
${\widehat f}_{t}^{\#} = U_{FG} \circ f_{t}^{\#} \circ U_{GF} : L^{2}
(Y)\rightarrow L^{2}(Y)$ is a fractal flow with infinitesimal generator $i
\widehat L = U_{FG}\circ L \circ U_{GF}$. \vskip 2mm

\begin{example}
[Vector field flow]\label{vector}Let $V:{\mathbb{R}}^{2}\rightarrow
{\mathbb{R}}^{2}$ be a 2-dimensional vector field given by $V(x,y)=(-y,x)$.
Define a flow $f:{\mathbb{R}}^{2}\times{\mathbb{R}}\rightarrow{\mathbb{R}}%
^{2}$, in the usual way by solving the autonomous system
\[
\frac{d}{dt}f(\mathbf{a},t)=V(f(\mathbf{a},t)),\qquad f(\mathbf{a}%
,0)=\mathbf{a}.
\]
The solution is, with notation $f_{t}(\mathbf{a})=f(\mathbf{a},t)$ and
$\mathbf{a}=(a,b)$,
\[
f_{t}(a,b)=(a\,\cos t+b\,\sin t,\;a\,\sin t-b\,\cos t).
\]
With $a,b$ fixed, as a function of $t$, the flow curves are circles centered
at the origin, so the domain of the flow $f_{t}(\mathbf{a})$ can be restricted
to $D\times{\mathbb{R}}$, where $D$ is the closed unit disk.

Now consider the area preserving fractal homeomorphism of
Example~\ref{ex:triangles}. Let $D$ be the largest inscribed disk in the
equilateral triangle $\bigtriangleup$. Without loss of generality, assume that
$D$ has radius $1$ and consider the flow $f_{t}(\mathbf{x})$ as in the
paragraph above. The fractally transformed flow, as in
Definition~\ref{def:flow}, is
\[
g_{t}(\mathbf{y})=T_{FG}\circ f_{t}\circ T_{GF}\,(\mathbf{y}),
\]
which as a function of $\mathbf{y}$, is area preserving. If the fractally
transformed vector field is denoted $\widehat{V}=T_{FG}\circ V\circ T_{GF}$,
then $g_{t}$ is the fractal flow of the vector field $\widehat{V}$. See Figure
\ref{wheels}.%
\begin{figure}[ptb]%
\centering
\includegraphics[
natheight=7.111400in,
natwidth=16.416700in,
height=1.8066in,
width=4.1321in
]%
{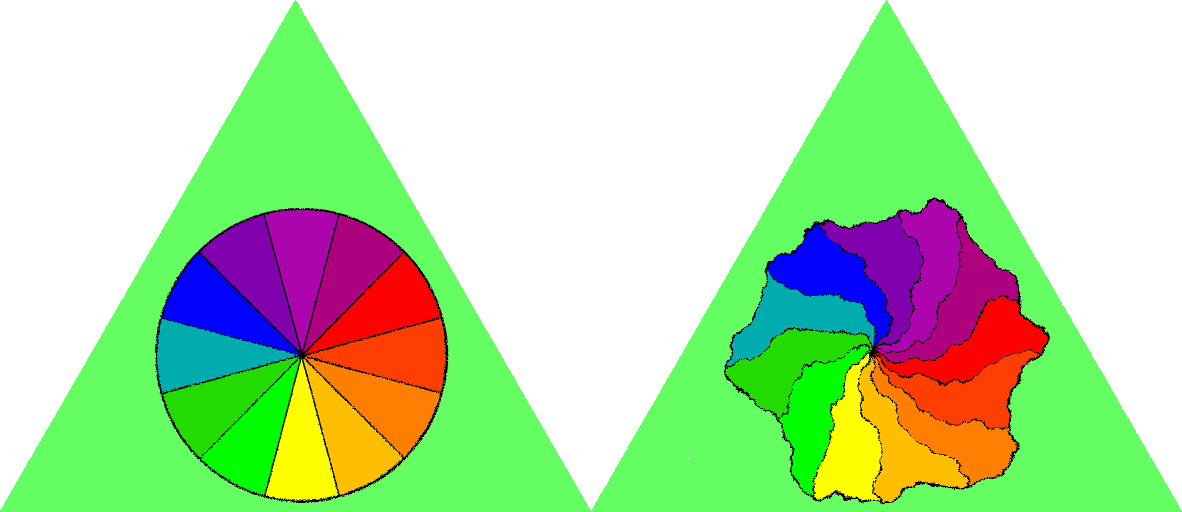}%
\caption{This image relates to Example \ref{wheels}.}%
\label{wheels}%
\end{figure}

\end{example}

\begin{example}
[Fractal flows on the unit interval and the circle]\label{ex:flow}

Consider the Lebesgue measure preserving flow on a line segment $[0,1]$ or the
circle $S^{1}$ defined by $f_{t}:[0,1]\rightarrow[0,1],$ $t\in(-\infty
,\infty),$ defined by
\[
x\mapsto(x+t)\operatorname{mod}1\text{.}%
\]

Consider any measure $\rho$ supported on $[0,1]$ that is absolutely continuous
with respect to Lebesque measure $\mu$. We may treat $\rho$ as a model for the
brightness and colours of a one-dimensional picture: the rate at which light
of a set of frequencies is emitted, or reflected, in unit time under steady
illumination by the Borel set $B$ is $\rho_{0}(B)$; see \cite{superfractals}.
A vector of measures $(\rho_{R},\rho_{G},\rho_{B})$ represents the red, green,
and blue components. With notation as above, $f_{t}^{\#} : \mathcal{M
}\rightarrow\mathcal{M}$ is a flow on $\mathcal{M}$. The orbit of a particular
measure $\rho_{0}$ models the picture being transported/translated at constant
velocity along the line segment (what comes out at one end of the line segment
immediately reenters the other end) or around the circle $S^{1}$.

Given an initial measure $\rho_{0}$, absolutely continuous with respect to
Lebesque measure, consider its orbit $\rho_{t} = f_{t}^{\#}(\rho_{0})$, i.e.
$\rho_{t}(B) = \rho_{0} ( f_{-t} B)$. Interpreted in the model, $\rho_{t}$ is
the translated picture/measure. By the Radon Nikodym theorem there is a
measurable function $\varrho_{0}$ such that
\[
\rho_{0}(B)=\int\limits_{B}\varrho_{0}(x) \, d\mu=\langle\chi_{B},\varrho_{0}
\rangle
\]
for all Borel sets $B$, where $\chi_{B}$ is the indicator function for $B$. It
follows that
\[
\rho_{t}(B)=\int\limits_{B}\varrho_{t}(x)\, d\mu\, =\, \langle\chi_{B}%
,\varrho_{t} \rangle\, =\, \langle f_{t}^{\#} \chi_{B},\varrho_{0} \rangle,
\]
where $\varrho_{t}(x) = f_{-t}^{\#} (\varrho_{0})(x) = \varrho_{0} ((x+t)
\mod
1)$, and the last equality by a change of variable. Letting $V_{t} :=
f_{t}^{\#} : L^{2}(S^{1}) \rightarrow L^{2}(S^{1})$ , by the comments prior to
this example, $f_{t}^{\#}=e^{itL}$, where $L$ is a self-adjoint operator. It
is well-known that, for this choice of the flow $f_{t}$, the operatort $L$ is
an extension of the differential operator $-i\frac{d}{dx}$ acting on
infinitely differentiable functions on $S^{1}$. Therefore, on an appropriate
domain,
\[
V_{t} =e^{itL} = e^{t\frac{d}{dx}}.
\]

Let $T_{FG}:[0,1]\rightarrow\lbrack0,1]$ be a uniform Lebesgue
measure-preserving fractal transformation, as considered in
Section~\ref{sec:ss}, and let $U_{FG}:L^{2}([0,1])\rightarrow L^{2}([0,1])$ be
the corresponding unitary transformation. Then the fractal flow
\[
\widehat{V}_{t}:={\widehat{f}}_{t}^{\#}=U_{FG}\,V_{t}\,U_{GF}%
\]
is again a strongly continuous one parameter unitary group generated by the
self-adjoint operator $\widetilde{L}:=U_{FG\,}L\,U_{GF}$.

Figure \ref{allshifts} illustrates a fractal flow on $[0,1]$ for the case of
$T_{FG_{1}}$ in Example \ref{ex:1} and Section~\ref{sec:ss}. The bottom strip
shows an initial function $\varphi$ on the interval $[0,1]$. In its orbit
$V_{t}(\varphi),\,t\geq0$, this picture slides to the right, colours going off
the right-hand end and coming on at the left end, cyclically (not in the
figure). From the top of the figure reading downwards, the successive strips
show the same orbit under the fractal flow ${\widehat{V}}_{t}$ at times
$t=0,1,2,...,7$. Then there is a white gap, followed by the flow at time
$t=100$. \vskip2mm

A surprising property of the flow orbit $\rho_{t}$ is that it is a continuous
function of $t$, although $U_{FG}$ may map continuous functions to
discontinuous ones. The proof is a consequence of the fact \cite[Proposition
2.5]{SS} that $\| f_{t}^{\#} \varrho- \varrho\|_{L^{1}} \rightarrow0$ as
$t\rightarrow0$.
\end{example}%

\begin{figure}[ptb]%
\centering
\includegraphics[
natheight=5.472500in,
natwidth=7.111400in,
height=2.0349in,
width=3.9851in
]%
{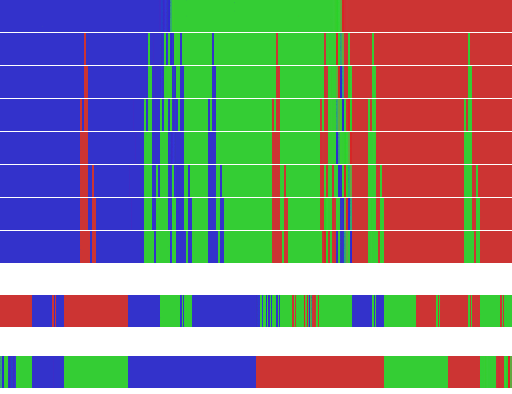}%
\caption{See text. Illustration of a fractal flow.}%
\label{allshifts}%
\end{figure}

\end{document}